\documentclass[11pt,reqno]{amsart}
\usepackage{charter}
\usepackage[T1]{fontenc}
\usepackage[latin9]{inputenc}
\usepackage{geometry}
\usepackage{color}
\usepackage{prettyref}
\usepackage{mathtools}
\usepackage{amsbsy}
\usepackage{amstext}
\usepackage{amsthm}
\usepackage{amssymb}
\usepackage{setspace}
\usepackage[numbers]{natbib}
\usepackage[unicode=true,
 bookmarks=true,bookmarksnumbered=false,bookmarksopen=false,
 breaklinks=false,pdfborder={0 0 1},backref=false,colorlinks=true]
 {hyperref}
\hypersetup{
 pdfauthor={Michael Brannan, Alexandru Chirvasitu and Ami Viselter},
 pdfstartview={XYZ null null 1.25},citecolor=OliveGreen,linkcolor=RoyalBlue,urlcolor=blue}

\makeatletter
\numberwithin{equation}{section}
\numberwithin{figure}{section}
\theoremstyle{plain}
\newtheorem{thm}{\protect\theoremname}[section]
\theoremstyle{definition}
\newtheorem{defn}[thm]{\protect\definitionname}
\theoremstyle{remark}
\newtheorem{rem}[thm]{\protect\remarkname}
\theoremstyle{plain}
\newtheorem{lem}[thm]{\protect\lemmaname}
\theoremstyle{plain}
\newtheorem{cor}[thm]{\protect\corollaryname}

\ifx\proof\undefined
\newenvironment{proof}[1][\protect\proofname]{\par
	\normalfont\topsep6\p@\@plus6\p@\relax
	\trivlist
	\itemindent\parindent
	\item[\hskip\labelsep\scshape #1]\ignorespaces
}{%
	\endtrivlist\@endpefalse
}
\providecommand{\proofname}{Proof}
\fi
\theoremstyle{plain}
\newtheorem{prop}[thm]{\protect\propositionname}

\newcommand\bG{{\mathbb G}}
\newcommand\bH{{\mathbb H}}

\newcommand\bK{{\mathbb K}}

\newcommand\bR{{\mathbb R}}

\newcommand\cD{{\mathcal D}}

\newcommand\cO{{\mathcal O}}

\newcommand\cT{{\mathcal T}}

\usepackage{eucal}
\let\mathcal=\CMcal
\usepackage{dsfont}
\usepackage{graphicx} 
\usepackage[all]{xy}
\usepackage[usenames,dvipsnames]{xcolor}
\usepackage{upgreek}

\newrefformat{eq}{\eqref{#1}}
\newrefformat{def}{Definition~\ref{#1}}
\newrefformat{lem}{Lemma~\ref{#1}}
\newrefformat{pro}{Proposition~\ref{#1}}
\newrefformat{prop}{Proposition~\ref{#1}}
\newrefformat{cor}{Corollary~\ref{#1}}
\newrefformat{thm}{Theorem~\ref{#1}}
\newrefformat{exa}{Example~\ref{#1}}
\newrefformat{rem}{Remark~\ref{#1}}
\newrefformat{sec}{Section~\ref{#1}}
\newrefformat{sub}{Subsection~\ref{#1}}
\newrefformat{conj}{Conjecture~\ref{#1}}
\newrefformat{desc}{\ref{#1}}
\newrefformat{ques}{Question~\ref{#1}}

\DeclareSymbolFont{YHlargesymbols}{OMX}{yhex}{m}{n}
\DeclareMathAccent{\wideparen}{\mathord}{YHlargesymbols}{"F3}

\DeclareMathOperator{\Mor}{Mor}
\DeclareMathOperator{\Inv}{Inv}

\providecommand{\corollaryname}{Corollary}
\providecommand{\definitionname}{Definition}

\providecommand{\lemmaname}{Lemma}

\providecommand{\propositionname}{Proposition}
\providecommand{\remarkname}{Remark}
\providecommand{\theoremname}{Theorem}

\makeatother

\providecommand{\definitionname}{Definition}
\providecommand{\lemmaname}{Lemma}
\providecommand{\propositionname}{Proposition}
\providecommand{\remarkname}{Remark}
\providecommand{\theoremname}{Theorem}

\renewcommand{\S}{Section~}

\begin{document}
\global\long\def\e{\varepsilon}%
\global\long\def\N{\mathbb{N}}%
\global\long\def\Z{\mathbb{Z}}%
\global\long\def\Q{\mathbb{Q}}%
\global\long\def\R{\mathbb{R}}%
\global\long\def\C{\mathbb{C}}%
\global\long\def\G{\mathbb{G}}%
\global\long\def\QG{\mathbb{G}}%
\global\long\def\QH{\mathbb{H}}%
\global\long\def\HH{\mathbb{H}}%
\global\long\def\bn{\mathbb{N}}%
\global\long\def\br{\mathbb{R}}%
\global\long\def\bc{\mathbb{C}}%
\global\long\def\bt{\mathbb{T}}%
\global\long\def\H{\EuScript H}%
\global\long\def\J{\mathcal{J}}%
\global\long\def\K{\mathcal{K}}%
\global\long\def\KHilb{\EuScript K}%
\global\long\def\a{\alpha}%
\global\long\def\be{\beta}%
\global\long\def\l{\lambda}%
\global\long\def\om{\omega}%
\global\long\def\z{\zeta}%
\global\long\def\gnsmap{\operatorname{\upeta}}%
\global\long\def\Aa{\mathcal{A}}%
\global\long\def\Aalg{\mathsf{A}}%
\global\long\def\Sant{\mathtt{S}}%
\global\long\def\Rant{\mathtt{R}}%
\global\long\def\Ree{\operatorname{Re}}%
\global\long\def\Img{\operatorname{Im}}%
\global\long\def\linspan{\operatorname{span}}%
\global\long\def\supp{\operatorname{supp}}%
\global\long\def\slim{\operatorname*{s-lim}}%
\global\long\def\clinspan{\operatorname{\overline{span}}}%
\global\long\def\co{\operatorname{co}}%
\global\long\def\pres#1#2#3{\prescript{#1}{#2}{#3}}%
\global\long\def\tensor{\otimes}%
\global\long\def\tensormin{\mathbin{\otimes_{\mathrm{min}}}}%
\global\long\def\tensorn{\mathbin{\overline{\otimes}}}%
\global\long\def\A{\forall}%
\global\long\def\i{\mathrm{id}}%
\global\long\def\tr{\operatorname{tr}}%
\global\long\def\one{\mathds{1}}%
\global\long\def\Ww{\mathds{W}}%
\global\long\def\wW{\text{\reflectbox{\ensuremath{\Ww}}}\:\!}%
\global\long\def\op{\mathrm{op}}%
\global\long\def\WW{{\mathds{V}\!\!\text{\reflectbox{\ensuremath{\mathds{V}}}}}}%
\global\long\def\Vv{\mathds{V}}%
\global\long\def\vV{\text{\reflectbox{\ensuremath{\Vv}}}\:\!}%
\global\long\def\M#1{\operatorname{M}(#1)}%
\global\long\def\Linfty#1{L^{\infty}(#1)}%
\global\long\def\Lone#1{L^{1}(#1)}%
\global\long\def\Lp#1{L^{p}(#1)}%
\global\long\def\Lq#1{L^{q}(#1)}%
\global\long\def\LoneSharp#1{L_{\sharp}^{1}(#1)}%
\global\long\def\Ltwo#1{L^{2}(#1)}%
\global\long\def\Cz#1{\mathrm{C}_{0}(#1)}%
\global\long\def\CzU#1{\mathrm{C}_{0}^{\mathrm{u}}(#1)}%
\global\long\def\CzUSSharp#1{\mathrm{C}_{0}^{\mathrm{u}}(#1)_{\sharp}^{*}}%
\global\long\def\CU#1{\mathrm{C}^{\mathrm{u}}(#1)}%
\global\long\def\Cb#1{\mathrm{C}_{b}(#1)}%
\global\long\def\CStarF#1{\mathrm{C}^{*}(#1)}%
\global\long\def\CStarR#1{\mathrm{C}_{\mathrm{r}}^{*}(#1)}%
\global\long\def\Cc#1{\mathrm{C}_{c}(#1)}%
\global\long\def\CC#1{\mathrm{C}(#1)}%
\global\long\def\CzX#1#2{\mathrm{C}_{0}^{#1}(#2)}%
\global\long\def\CcX#1#2{\mathrm{C}_{c}^{#1}(#2)}%
\global\long\def\CcTwo#1{\mathrm{C}_{c}^{2}(#1)}%
\global\long\def\aa#1{\mathfrak{a}_{#1}}%
\global\long\def\linfty#1{\ell^{\infty}(#1)}%
\global\long\def\lone#1{\ell^{1}(#1)}%
\global\long\def\ltwo#1{\ell^{2}(#1)}%
\global\long\def\cz#1{\mathrm{c}_{0}(#1)}%
\global\long\def\Pol#1{\mathrm{Pol}(#1)}%
\global\long\def\Ltwozero#1{L_{0}^{2}(#1)}%
\global\long\def\Irred#1{\mathrm{Irred}(#1)}%
\global\long\def\conv{\star}%
\global\long\def\Ad#1{\mathrm{Ad}(#1)}%
\global\long\def\VN#1{\mathrm{VN}(#1)}%
\global\long\def\d{\,\mathrm{d}}%
\global\long\def\t{\mathrm{t}}%
\global\long\def\tie#1{\wideparen{#1}}%
\global\long\def\caret{\text{\textasciicircum}}%
\global\long\def\epp#1{#1{}_{+}^{\caret}}%
\global\long\def\tp{\mathbin{\xymatrix{*+<.7ex>[o][F-]{\scriptstyle \top}}
 } }%
\global\long\def\tpsmall{\mathbin{\xymatrix{*+<.5ex>[o][F-]{\scriptscriptstyle \top}}
 } }%
\global\long\def\tpr{\mathbin{\xymatrix{*+<.7ex>[o][F-]{\scriptstyle \bot}}
 } }%
\global\long\def\tprsmall{\mathbin{\xymatrix{*+<.5ex>[o][F-]{\scriptscriptstyle \bot}}
 } }%


\newcommand{\define}[1]{{\em #1}}

\newcommand{\cat}[1]{\textsc{#1}}

\renewcommand{\square}{\mathrel{\Box}}

\newcommand{\mike}[1]{{\bf \color{red} **{#1}**}}
\newcommand{\alex}[1]{{\bf \color{green} **{#1}**}}
\newcommand{\ami}[1]{{\bf \color{cyan} **{#1}**}}

\newcommand\Section[1]{\section{#1}\setcounter{lemma}{0}}

\newcommand\li[1]{\Linfty{#1}}

\title{Actions, quotients and lattices of locally compact quantum groups}

\author{Michael Brannan}
\address{Department of Mathematics, Texas A\&M University, College Station,
TX 77843-3368, USA}
\email{mbrannan@math.tamu.edu}
\author{Alexandru Chirvasitu}
\address{Department of Mathematics, University at Buffalo, Buffalo, NY 14260-2900,
USA}
\email{achirvas@buffalo.edu}
\author{Ami Viselter}
\address{Department of Mathematics, University of Haifa, 31905 Haifa, Israel}
\email{aviselter@univ.haifa.ac.il}

\begin{abstract}
  We prove a number of property (T) permanence results for locally compact quantum groups under exact sequences and the presence of invariant states, analogous to their classical versions. Along the way we characterize the existence of invariant weights on quantum homogeneous spaces of quotient type, and relate invariant states for LCQG actions on von Neumann algebras to invariant vectors in canonical unitary implementations, providing an application to amenability. Finally, we introduce a notion of lattice in a locally compact quantum group, noting examples provided by Drinfeld doubles of compact quantum groups. We show that property (T) lifts from a lattice to the ambient LCQG, just as it does classically, thus obtaining new examples of non-classical, non-compact, non-discrete LCQGs with property (T).
\end{abstract}

\maketitle


\vspace{.5cm}


\tableofcontents

\section*{Introduction}

The axiomatization of locally compact quantum groups (or LCQGs for short) has reached a stable state with the advent of \cite{Kustermans_Vaes__LCQG_C_star,Kustermans_Vaes__LCQG_von_Neumann,Kustermans__LCQG_universal, wor-mult-wuo,wor-mult-vLr,mnw,wor-mult-PxG} (based on earlier work such as \cite{bs-bYv}), making the field a rich source of examples, questions and problems pertaining to representation theory, operator algebras, geometric group theory and affiliated subjects.

The present paper is motivated primarily by prior work on property (T) for LCQGs. In various degrees of generality, in the quantum setting this representation-theoretic rigidity property has been studied in a number of sources: \cite{Petrescu_Joita__prop_T_Kac_alg} for Kac algebras, \cite{Bedos_Conti_Tuset__amen_co_amen_alg_QGs_coreps} in the algebraic setting, \cite{Fima__prop_T,kyed-coh,Kyed_Soltan__prop_T_exotic_QG_norms} for {\it discrete} quantum groups and finally \cite{Daws_Fima_Skalski_White_Haagerup_LCQG,Daws_Skalski_Viselter__prop_T,Brannan_Kerr__QG_T_WM} in full generality, for LCQGs. A number of papers address the problem of constructing examples of LCQGs with property (T), e.g.~\cite{Arano__unit_sph_rep_Drinfeld_dbls,Fima_Mukherjee_Patri_compact_bicross_prod,Vaes_Valvekens__prop_T_DQG_subfactor_tri_rep}.

We are concerned here with the ``hereditary'' character of property (T), i.e.~its preservation under passing to appropriate subgroups, quotients, extensions, etc. One familiar result is that for a short exact sequence
\begin{equation*}
  1\to H\to G\to G/H\to 1
\end{equation*}
of locally compact groups, $G$ has property (T) if and only if both $G/H$ and the pair $(G,H)$ do (see for instance \cite[Exercise 1.8.12]{Bekka_de_la_Harpe_Valette__book} or the somewhat weaker version in \cite[Lemma 7.4.1]{zim-erg}). The corresponding result for {\it discrete} quantum groups was proven in \cite{Bhattacharya_Brannan_Chirvasitu_Wang__rig_soft_prop_resid_fin_DQG}, and we generalize it here in full for arbitrary short exact sequences of LCQGs in \prettyref{sec:property-t-exact}.

Another celebrated classical result with deep ramifications is the equivalence, for a locally compact group $G$, of property (T) for $G$ and for any of its {\it finite-covolume} closed subgroups $H\le G$, i.e.~those for which the homogeneous space $G/H$ admits a finite $G$-invariant regular Borel measure \cite[Theorem 1.7.1]{Bekka_de_la_Harpe_Valette__book}. This affords deducing property (T) for certain discrete groups realizable as {\it lattices} (closed, discrete, finite-covolume subgroups) of Lie groups.

We prove an analog of this finite-covolume permanence result in \prettyref{thm:covol} below: property (T) lifts from finite-covolume closed quantum subgroups. Given that, by \prettyref{thm:dbl-inv}, a Kac-type discrete quantum group $\widehat{\bG}$ is a lattice in the Drinfeld double $\cD\bG$ of the corresponding compact quantum group $\bG$, this provides examples of non-discrete quantum groups $\cD\bG$ with property (T).


Along the way towards the above-mentioned property (T) permanence statements we prove a number of auxiliary results that we hope might have wider applicability as general-purpose tools in dealing with restrictions of unitary representations to closed quantum subgroups $\bH\le \bG$ of LCQGs, and also with invariance properties for measures on quantum homogeneous spaces $\bG/\bH$. The unifying thread throughout is that of LCQG actions on von Neumann algebras, with canonical unitary implementations and invariant weights / states playing a central role in the discussion. 

A more detailed, albeit brief summary of the contents of the paper follows.

In \prettyref{sec:prelim} we gather some of the requisite preliminary material on the structure of locally compact quantum groups. 

The main result of \prettyref{sec:property-t-exact} is \prettyref{thm:ext}, the analog of the classical result stating that given a closed normal subgroup $\bH\trianglelefteq \bG$ of a locally compact group, property (T) for $\bG$ is equivalent to property (T) for the pair $(\bG,\bH)$ and the quotient $\bG/\bH$ (see e.g.~\cite[Exercise 1.8.12]{Bekka_de_la_Harpe_Valette__book}).

In the process of proving that result we provide, in \prettyref{prop:fact}, a characterization of unitary representations of $\bG$ {\it factoring through} a quotient $\bG/\bH$ for normal $\bH\trianglelefteq \bG$. Although relatively simple and unsurprising, we were not able to find the remark in the literature; it would presumably be of some independent interest in its own right.

In \prettyref{sec:invariant-weights} we turn to invariant measures on homogeneous spaces $\bG/\bH$. For LCQGs this translates to invariant normal semi-finite faithful weights on the von Neumann algebra $\li{\bG/\bH}$. The main result of the section is \prettyref{thm:quot_inv_nsf_weight}, where we prove the analog of the classical characterization of inclusions $\bH\le \bG$ for which $\bG/\bH$ has a $\bG$-invariant measure: they are precisely those for which the modular function of $\bG$ restricts to the modular function of $\bH$ \cite[Corollary B.1.7]{Bekka_de_la_Harpe_Valette__book}.

The invariant measure theme recurs in \prettyref{sec:canon-impl-invar}, where we prove in \prettyref{thm:inv-can} that given an action by a LCQG $\bG$ on a von Neumann algebra $N$, a normal state on $N$ is invariant under the action if and only if the corresponding vector is invariant under the canonical unitary implementation of the action. This result was previously known to hold for \emph{discrete} quantum groups \cite[Proposition 4.11]{Daws_Skalski_Viselter__prop_T}. It has some applications to characterizing a strong form of amenability for unitary representations by means of invariant vectors, as we discuss in \prettyref{thm:amnbl-chars}. 

\prettyref{sec:fin_covol_lat} revolves around closed quantum subgroups $\HH \le \G$ of \emph{finite covolume}, that is, such that $\bG/\bH$ admits a $\bG$-invariant normal state.
We prove in \prettyref{thm:finite-covol-unim} that in this case, the unimodularity of $\bH$ and $\bG$ are equivalent.
We then define {\it lattices} of locally compact quantum groups as discrete closed quantum subgroups of finite covolume (see \prettyref{def:lat}). By \prettyref{thm:dbl-inv}, examples include the discrete ``halves'' of Drinfeld doubles of Kac-type compact quantum groups. In \prettyref{thm:covol} we obtain the quantum counterpart of the result that property (T) lifts along inclusions with finite covolume, and transfers from lattices to the ambient LCQGs.

\section{Preliminaries}\label{sec:prelim}

All Hilbert spaces in the paper are complex, and the inner products are linear in the left variable. For a Hilbert space $\H$ and $\z,\eta\in\H$, let $\om_{\z,\eta} \in B(\H)_*$ be defined by $B(\H) \ni T \mapsto \langle T\z,\eta \rangle$, and set $\om_\z := \om_{\z,\z}$.
Representations of C$^{*}$-algebras will always be assumed non-degenerate, and the units will be denoted by $\one$ when they exist. The symbol $\tensor$ is reserved for the tensor product of Hilbert spaces and maps, $\tensormin$ stands for the minimal tensor product of C$^{*}$-algebras, and $\tensorn$ is designated for the normal tensor product of von Neumann algebras. For a C$^{*}$-algebra $A$, write $\M{A}$ for its multiplier algebra. For C$^{*}$-algebras $A$ and $B$, a \emph{morphism} from $A$ to $B$ is a  $*$-homomorphism $\Phi:A \to \M{B}$ that is non-degenerate, i.e., $\Phi(A)B$ is total in $B$. The set of all such morphisms is denoted by $\Mor(A,B)$.

We assume familiarity with modular theory of von Neumann algebras \citep{Stratila__mod_thy,Takesaki__Tomita_thy,Takesaki__book_vol_2}, including the theory of operator-valued weights (originally \citep{Haagerup__oper_val_weights_1,Haagerup__oper_val_weights_2}). The extended positive part of a von Neumann algebra $M$ is denoted by $\widehat{M}_{+}$ or $\epp{M}$. For a normal semi-finite faithful (n.s.f.)~weight $\varphi$ on $M$, let 
\[
  \mathcal{M}_{\varphi}^{+} := \{x \in M_+ : \varphi(x) < \infty\}, \quad \mathcal{M}_{\varphi} := \linspan \mathcal{M}_{\varphi}^{+}, \quad \mathcal{N}_{\varphi} := \{x \in M : x^*x \in \mathcal{M}_{\varphi}^{+}\},
\]
and denote the GNS Hilbert space, GNS map, modular conjugation, modular operator and modular automorphism group of $\varphi$ by $\Ltwo{M,\varphi}$, $\gnsmap_\varphi$, $J_\varphi$, $\nabla_\varphi$ and $\left(\sigma_{t}^\varphi\right)_{t\in\R}$, respectively, and write $T_\varphi:=J_\varphi\nabla_\varphi^{1/2}$.

Unless otherwise indicated, the following preliminaries on locally compact quantum groups are taken from \cite{Kustermans_Vaes__LCQG_C_star,Kustermans_Vaes__LCQG_von_Neumann,Kustermans__LCQG_universal}. They are far from being exhaustive, and we refer to the original articles for more details.

\begin{defn}\label{def:lcqg}
  A \emph{locally compact quantum group} (in short, LCQG) is a pair $\G = (M,\Delta)$ such that:
  \begin{enumerate}
  \item $M$ is a von Neumann algebra;
  \item $\Delta$, called the \emph{co-multiplication} of $\G$, is a unital normal $*$-homomorphism from $M$ to $M \tensorn M$ that is co-associative: $(\Delta \tensor \i) \Delta = (\i \tensor \Delta) \Delta$;
  \item $M$ admits n.s.f.~weights $\varphi,\psi$, called the \emph{Haar weights}, that are left and right invariant, respectively, in the sense that
    \begin{gather*}
      \varphi((\om \tensor \i)(\Delta(x))) = \varphi(x)\om(\one) \qquad (\forall x \in \mathcal{M}_{\varphi}^{+}, \om \in M_*^+), \\
      \psi((\i \tensor \om)(\Delta(x))) = \psi(x)\om(\one) \qquad (\forall x \in \mathcal{M}_{\psi}^{+}, \om \in M_*^+).
    \end{gather*}
  \end{enumerate}
  We set $\Linfty{\G} := M$, $\Lone{\G} := M_*$ and $\Ltwo{\G} := \Ltwo{M,\varphi}$.
\end{defn}

The easiest example of LCQGs comes from locally compact groups $G$: indeed, just take the usual $\Linfty{G}$ with the co-multiplication $\Delta : \Linfty{G} \to \Linfty{G \times G} \cong \Linfty{G} \tensorn \Linfty{G}$ given by $(\Delta(f))(s,t) := f(st)$ for $f \in \Linfty{G}$ and $s,t \in G$.

Every LCQG $\G$ has a \emph{dual} LCQG, denoted by $\widehat{\G}$. We will not explain here how this duality works, but mention the double dual property: $\widehat{\widehat{\G}} = \G$, and the fact that this duality extends Pontryagin's duality for locally compact abelian groups. Elements pertaining to the dual $\widehat{\G}$ will be decorated with a hat, e.g.~$\widehat{\varphi}$. Remark that we can and will identify $\Ltwo{\widehat{\G}}$ with $\Ltwo{\G}$. There exists a (multiplicative) unitary $W \in \Linfty{\bG} \tensorn \Linfty{\widehat{\G}}$, called the \emph{left regular representation} of $\G$. It implements $\Delta$ in the sense that $\Delta(x) = W^* (\one \tensor x) W$ for all $x \in \Linfty{\G}$. Furthermore, we have $\widehat{W} = \sigma(W^*)$, where $\sigma$ is the flip automorphism. 

There are also two C$^*$-algebraic ``pictures'' of a LCQG $\G$. The set
\begin{equation*}
  \Cz{\G} := \overline{\{(\i \tensor \om)(W) : \om \in \Lone{\widehat{\G}}\}}
\end{equation*}
is a WOT-dense C$^*$-subalgebra of $\Linfty{\G}$, called the \emph{reduced} C$^*$-algebra of $\G$. Considering $W$ as acting on $\Ltwo{\G} \tensor \Ltwo{\G}$, it belongs to both $\M{\Cz{\G} \tensormin K(\Ltwo{\G})}$ and $\M{\Cz{\G} \tensormin \Cz{\widehat{\G}}}$. Furthermore, $\Delta$ restricts to an element of $\Mor(\Cz{\G}, \Cz{\G} \tensormin \Cz{\G})$. The unitary antipode, as a $*$-anti-automorphism of either $\Cz{\G}$ or $\Linfty{\G}$, will be denoted by $R$. There is also the \emph{universal} C$^*$-algebra $\CzU{\G}$ of $\G$ with its own co-multiplication $\Delta^{\mathrm{u}} \in \Mor(\CzU{\G}, \CzU{\G} \tensormin \CzU{\G})$. It admits a universality property related to representations; see below. The \emph{reducing morphism}, which is a surjective $*$-homomorphism $\Lambda : \CzU{\G} \to \Cz{\G}$, satisfies $(\Lambda \tensor \Lambda) \circ \Delta^{\mathrm{u}} = \Delta|_{\Cz{\G}} \circ \Lambda$. There are also the universal version $\WW \in \M{\CzU{\G} \tensormin \CzU{\widehat{\G}}}$, and the two semi-universal versions $\wW,\Ww$, of $W$. For instance, we have $\Ww \in \M{\CzU{\G} \tensormin \Cz{\widehat{\G}}}$. 

A unitary representation, or simply a \emph{representation}, of a LCQG $\G$ on a Hilbert space $\H$ is a unitary operator $U\in\Linfty{\G}\tensorn B(\H)$ satisfying $(\Delta\tensor\i)(U)=U_{13}U_{23}$, where the subscript indicates tensor product leg numbering. In fact, we automatically have $U\in\M{\Cz{\G}\tensormin K(\H)}$. There is a bijection between the representations of $\G$ and the representations of the C$^{*}$-algebra $\CzU{\widehat \G}$ associating $\Phi \in \Mor(\CzU{\widehat \G}, K(\H))$ to $(\i \tensor \Phi)(\wW) \in \M{\Cz{\G}\tensormin K(\H)}$, for a Hilbert space $\H$. From the dual side, the left regular representation $\widehat{W}$ of $\widehat{\G}$ and the trivial representation $\one \in \Linfty{\widehat{\G}}$ of $\widehat{\G}$ correspond to the reducing morphism $\Lambda$ and to the \emph{co-unit} $\epsilon$ of $\G$.

Let $U,V$ be representations of a LCQG $\G$ on Hilbert spaces $\H,\KHilb$, respectively. The \emph{contragradient} of $U$ is the representation $\overline{U} := (R \tensor \top)(U)$ of $\G$ on $\overline{\H}$, where $\overline{\H}$ is the (complex) conjugate Hilbert space of $\H$ and $\top:B(\H) \rightarrow B(\overline{\H})$ is the transpose map, defined by
$\top(x)(\overline{\xi}) = \overline{ x^*(\xi) }$ for $x \in B(\H)$ and $\xi \in \H$. We can \emph{tensor} $U$ and $V$ in two ways, yielding the following representations of $\G$ on $\H \tensor \KHilb$:
\[
U \tp V := U_{12} V_{13} \text{ and } U \tpr V := V_{13} U_{12}
\]
(we warn the reader that the meaning of the notation $\tpr$ is not consistent with the one in \cite{Woronowicz__CMP}). We have $\overline{\overline{U}} = U$ and $\overline{U \tp V} = \overline{U} \tpr \overline{V}$ (identifying $\overline{\H \tensor \KHilb} \cong \overline{\H} \tensor \overline{\KHilb}$).

Actions are of basic importance in this paper. A \emph{left} (resp., \emph{right}) \emph{action} of a LCQG $\G$ on a von Neumann $N$ is an injective normal unital $*$-homomorphism $\a : N \to \Linfty{\G} \tensorn N$ (resp., $\a : N \to N \tensorn \Linfty{\G}$) such that $(\i \tensor \a)  \a = (\Delta \tensor \i) \a$ (resp., $(\a \tensor \i)\a = (\i \tensor \Delta)\a$).

We require some material on homomorphisms between LCQGs from \cite[\S 12]{Kustermans__LCQG_universal}, \cite{Meyer_Roy_Woronowicz__hom_quant_grps} and \cite[Subsection 1.3]{Daws_Kasprzak_Skalski_Soltan__closed_q_subgroups_LCQGs} (note the different conventions regarding $W$ being the left/right regular representation). For LCQGs $\G,\HH$, there is a 1-1 correspondence between the following classes of objects:
\begin{enumerate}
\item strong quantum homomorphisms: elements $\pi \in \Mor(\CzU{\G},\CzU{\HH})$ that intertwine the co-multiplications: $(\pi \tensor \pi) \circ \Delta_{\G}^\mathrm{u} = \Delta_{\HH}^\mathrm{u} \circ \pi$;
\item left quantum homomorphisms: elements $\rho_l \in \Mor(\Cz{\G}, \Cz{\HH} \tensormin \Cz{\G})$ satisfying $(\i \tensor \Delta_{\G}) \circ \rho_l = (\rho_l \tensor \i) \circ \Delta_{\G}$ and $(\Delta_{\HH} \tensor \i) \circ \rho_l = (\i \tensor \rho_l) \circ \rho_l$;
\item right quantum homomorphisms: elements $\rho_r \in \Mor(\Cz{\G}, \Cz{\G} \tensormin \Cz{\HH})$ satisfying $(\Delta_{\G} \tensor \i) \circ \rho_r = (\i \tensor \rho_r) \circ \Delta_{\G}$ and $(\i \tensor \Delta_{\HH}) \circ \rho_r = (\rho_r \tensor \i) \circ \rho_r$.

\end{enumerate}
These objects describe a \emph{homomorphism from $\HH$ to $\G$}. In fact, $\rho_l$, resp.~$\rho_r$, extends (uniquely) to a left, resp.~right, action of $\HH$ on the von Neumann algebra $\Linfty{\G}$, and $\pi,\rho_l,\rho_r$ are related to one another by the identities
\begin{align}
	\rho_l \circ \Lambda_{\G} & = ((\Lambda_{\HH} \circ \pi) \tensor \Lambda_{\G}) \circ \Delta^{\mathrm{u}}_\G, \label{eq:left_quant_hom__comult} \\
	\rho_r \circ \Lambda_{\G} & = (\Lambda_{\G} \tensor (\Lambda_{\HH} \circ \pi)) \circ \Delta^{\mathrm{u}}_\G, \label{eq:right_quant_hom__comult} \\
	\rho_l & = (R_\HH \tensor R_\G) \circ \sigma \circ \rho_r \circ R_\G \label{eq:left_right_quant_homs}
\end{align}
(in particular, see \cite[Theorem 5.3 equation (33) and Theorem 5.5 equation (35)]{Meyer_Roy_Woronowicz__hom_quant_grps}). 
Every strong quantum homomorphism $\pi \in \Mor(\CzU{\G},\CzU{\HH})$ from $\HH$ to $\G$ has a \emph{dual} strong quantum homomorphism $\widehat{\pi} \in \Mor(\CzU{\widehat{\HH}},\CzU{\widehat{\G}})$ from $\widehat{\G}$ to $\widehat{\HH}$, which satisfies 
\begin{equation}
(\pi \tensor \i)(\WW_\G) = (\i \tensor \widehat{\pi})(\WW_\HH). \label{eq:dual_str_quant_hom}
\end{equation}
Another useful identity is
\begin{equation}
\epsilon_\HH \circ \pi = \epsilon_\G. \label{eq:counit_str_quant_hom}
\end{equation}

The following definitions and results are from \cite{Daws_Kasprzak_Skalski_Soltan__closed_q_subgroups_LCQGs}. Let again $\G,\HH$ be LCQGs. 
We say that \emph{$\HH$ is a closed quantum subgroup of $\G$} in the sense of Woronowicz, resp.~Vaes, if there exists a strong quantum homomorphism $\pi$ from $\HH$ to $\G$ such that $\pi(\CzU{\G}) = \CzU{\HH}$, resp.~if there exists a normal injective $*$-homomorphism $\gamma : \Linfty{\widehat{\HH}} \to \Linfty{\widehat{\G}}$ intertwining the co-multiplications: $(\gamma \tensor \gamma) \circ \Delta_{\widehat{\HH}} = \Delta_{\widehat{\G}} \circ \gamma$. The latter condition implies the former, and the associated strong quantum homomorphism $\pi$ satisfies
\begin{equation}
\gamma|_{\Cz{\widehat{\HH}}} \circ \Lambda_{\widehat{\HH}} = \Lambda_{\widehat{\G}} \circ \widehat{\pi}. \label{eq:Woronowicz_Vaes}
\end{equation}

Let $\HH$ be a closed quantum subgroup of a LCQG $\G$ in the sense of Woronowicz, and denote by $\a_l:\Linfty{\G}\to\Linfty{\HH}\tensorn\Linfty{\G}$ the left action of $\HH$ on $\Linfty{\G}$ (the extension of the suitable map $\rho_l$ above). The ``$L^{\infty}$ algebra of the quantum homogeneous space $\HH \backslash \G$'' is the fixed-point von Neumann algebra $\Linfty{\HH \backslash \G}:=\{x\in\Linfty{\G}:\a_l(x)=\one\tensor x\}$.  It is a right coideal, namely $\Delta_{\G}(\Linfty{\HH \backslash \G})\subseteq\Linfty{\HH \backslash \G}\tensorn\Linfty{\G}$.  Hence, $\Delta_{\G}$ restricts to a right action $\Delta_{\HH \backslash \G}$ of $\G$ on $\Linfty{\HH \backslash \G}$. Similarly, there is a right action $\a_r : \Linfty{\G}\to\Linfty{\G}\tensorn\Linfty{\HH}$ of $\HH$ on $\Linfty{\G}$, and its fixed-point algebra is denoted by $\Linfty{\G / \HH}$.

\section{Property (T) and exact sequences}\label{sec:property-t-exact}

The main result of this section is the following generalization of \cite[\S 1.7, pp.~63--64; see Exercise 1.8.12]{Bekka_de_la_Harpe_Valette__book}, which extends \cite[Proposition 4.13]{Bhattacharya_Brannan_Chirvasitu_Wang__rig_soft_prop_resid_fin_DQG} from discrete to locally compact quantum groups. The required notions will be introduced subsequently.

\begin{thm}\label{thm:ext}
  Let $\bG$ be a LCQG and $\bH\trianglelefteq \bG$ a normal closed quantum subgroup. Then $\bG$ has property (T) if and only if both $\bG/\bH$ and the pair $(\bG,\bH)$ have property (T). 
\end{thm}

The implication `$\implies$' is not new, as will be explained below.
Let us recall the definition of normality of closed quantum subgroups. Let $\G$ be a LCQG and $\HH$ a closed quantum subgroup in the sense of Vaes with associated embedding $\gamma : \Linfty{\widehat{\HH}} \to \Linfty{\widehat{\G}}$.

\begin{defn}[{\cite[Definition 2.10]{Vaes_Vainerman__low_dim_LCQGs}}]
 We say that \emph{$\HH$ is normal in $\G$} if $W_{\widehat{\G}}(\gamma(\Linfty{\widehat{\HH}}) \tensor \one)W_{\widehat{\G}}^* \subseteq \gamma(\Linfty{\widehat{\HH}}) \tensorn B(\Ltwo{\G})$.
\end{defn}

The following characterizations of normality will be used tacitly in the sequel.

\begin{thm}[{\cite[\S 4]{Kasprzak_Soltan__QG_with_proj_and_ext_of_LCQG}, originally \cite[Theorem 2.11]{Vaes_Vainerman__low_dim_LCQGs}}]
	The following conditions are equivalent:
	\begin{enumerate}
		\item $\HH$ is normal in $\G$;
		\item $\li{\bG/\bH} = \li{\bH\backslash\bG}$;
		\item $\Delta_{\G}(\li{\bG/\bH}) \subseteq \li{\bG/\bH} \tensorn \li{\bG/\bH}$.
	\end{enumerate}
	When these are satisfied, $(\Linfty{\bG/\bH}, (\Delta_{\G})|_{\Linfty{\bG/\bH}})$ is a LCQG, which we denote by $\bG/\bH$.
\end{thm}

By the last sentence in the theorem's statement, the LCQG $\widehat{\bG/\bH}$ naturally becomes a closed quantum subgroup of $\widehat{\G}$ in the sense of Vaes.

\subsection{Invariance, almost invariance and preservation}
Let $U$ be a representation of a LCQG $\G$ on a Hilbert space $\H$ with associated morphism $\Phi \in \Mor(\CzU{\widehat \G}, K(\H))$. View $U$ as acting on $\Ltwo{\G} \tensor \H$ when appropriate.

\begin{defn}[{\cite[\S 3]{Daws_Fima_Skalski_White_Haagerup_LCQG}}] \mbox{}
	\begin{enumerate}
		\item A vector $\z\in\H$ \emph{is invariant under} $U$ if $\Phi(a)\z = \widehat{\epsilon}(a)\z$ for all $a \in \CzU{\widehat{\G}}$, or, equivalently, if $U(\eta \tensor \z) = \eta \tensor \z$ for all $\eta \in \Ltwo{\G}$. The closed subspace of all such vectors in $\H$ is denoted by $\Inv (U)$.
		\item We say that $U$ \emph{has almost-invariant vectors} if there exists a net $(\z_i)_{i \in \mathcal{I}}$ of unit vectors in $\H$ such that $\Phi(a)\z_i - \widehat{\epsilon}(a)\z_i \xrightarrow[i \in \mathcal{I}]{} 0$ for all $a \in \CzU{\widehat{\G}}$, or, equivalently, such that $U(\eta \tensor \z_i) - \eta \tensor \z_i \xrightarrow[i \in \mathcal{I}]{} 0$ for all $\eta \in \Ltwo{\G}$.
	\end{enumerate}	
\end{defn}

\begin{defn}
A closed subspace $\H_0$ of $\H$ \emph{is preserved by} (or is globally invariant under) $U$ if the projection $P$ of $\H$ onto $\H_0$ satisfies $(\one \tensor P)U(\one \tensor P) = U(\one \tensor P)$.
\end{defn}
Under the assumptions of the last definition, the operator $U(\one \tensor P) \in \Linfty{\G} \tensorn B(\H_0)$ is unitary, that is: $U$ and $\one \tensor P$ commute, by \cite[Corollary 4.16]{Brannan_Daws_Samei__cb_rep_of_conv_alg_of_LCQGs}, and is thus a representation of $\G$ on $\H_0$, indeed---a sub-representation of $U$.

\subsection{Restrictions of representations}

We first need some preliminaries on restricting representations to closed quantum subgroups. Throughout this subsection we let $\G$ be a LCQG and $\HH$ be a closed quantum subgroup in the sense of Woronowicz. Let $\pi \in \Mor (\CzU{\G}, \CzU{\HH})$ be the associated strong quantum homomorphism and $\widehat{\pi} \in \Mor(\CzU{\widehat{\HH}}, \CzU{\widehat{\G}})$ be its dual. We use the notation from \prettyref{sec:prelim}, and in particular the actions $\a_r,\a_l$. 

\begin{defn}\label{def:triv}Let $U$ be a representation of $\G$ on a Hilbert space $\H$. Write
  \begin{equation*}
    \Phi \in \Mor(\CzU{\widehat \G}, K(\H))
  \end{equation*}
  for the associated morphism. 
\begin{enumerate}
	\item The \emph{restriction} of $U$ to $\HH$ is the representation $U|_\HH$ of $\HH$ on $\H$ whose corresponding morphism is $\Phi \circ \widehat{\pi}$. Equivalently (by \eqref{eq:dual_str_quant_hom}), $$U|_\HH = (\i \tensor (\Phi \circ\widehat{\pi}))(\wW_\HH) = ((\Lambda_{\HH} \circ \pi) \tensor \Phi)(\WW_\G).$$
	\item The elements of $\Inv (U|_\HH)$ will be called the \emph{$\HH$-invariant vectors of $U$}.
	\item We say that {\it $U$ is trivial  on $\HH$} if $U|_\HH$ is the trivial representation of $\HH$ on $\H$, namely the unit of $\Linfty{\HH} \tensorn B(\H)$; equivalently: $\Phi \circ \widehat{\pi} = \epsilon_{\widehat{\bH}}(\cdot)\one$.
	\item Suppose that $\bH \le \bG$ is normal. We say that $U$ {\it factors through $\bG\to \bG/\bH$} if $U \in \Linfty{\G / \HH} \tensorn B(\H)$.
\end{enumerate}	
\end{defn}

\begin{rem}\label{rem:rstr_inv_vects}
A vector in $\H$ that is invariant under $U$ is also invariant under $U|_\HH$; and a similar statement about almost-invariant vectors also holds. This is because $\epsilon_{\widehat{\bG}} \circ \widehat{\pi} = \epsilon_{\widehat{\bH}}$ (see \eqref{eq:counit_str_quant_hom}).
\end{rem}

\begin{defn}
	The pair $(\G,\HH)$ is said to have \emph{property (T)} if for every representation of $\G$ with almost-invariant vectors, its restriction to $\HH$ has a non-zero invariant vector.
\end{defn}

Evidently, $\G$ itself has property (T) \cite[\S 6]{Daws_Fima_Skalski_White_Haagerup_LCQG} if and only if the pair $(\G,\G)$ has property (T); and in this case, the pair $(\G,\bK)$ has property (T) for every closed quantum subgroup $\bK$ of $\G$.

For the rest of this subsection we fix $U,\H,\Phi$ as in \prettyref{def:triv}.

\begin{lem}\label{lem:rstr}
We have 
\begin{gather}
\label{eq:rstr_1} (\alpha_l\otimes \mathrm{id})(U) = (U|_\HH)_{13}U_{23} \in \li{\bH} \tensorn \li{\bG} \tensorn B(\H), \\
\label{eq:rstr_2} (\alpha_r\otimes\mathrm{id})(U) = U_{13}(U|_\HH)_{23} \in \li{\bG} \tensorn \li{\bH} \tensorn B(\H).
\end{gather}
Therefore, the representation $U$ is trivial on $\HH$ if and only if $(\alpha_l\otimes \mathrm{id})(U) = U_{23}$, if and only if $(\alpha_r\otimes\mathrm{id})(U) = U_{13}$.
\end{lem}
\begin{proof}
  The left hand side of \eqref{eq:rstr_1} equals $((\a_l \circ \Lambda_{\G}) \tensor \Phi)(\WW_\G)$, so by \eqref{eq:left_quant_hom__comult} it is obtained by applying $(\Lambda_{\bH}\circ\pi) \otimes \Lambda_{\G} \otimes \Phi$ to $(\Delta^{\mathrm{u}}_\G\otimes \mathrm{id})(\WW^{\bG})$. Since the latter is $\WW^{\bG}_{13}\WW^{\bG}_{23}$, we obtain equation \eqref{eq:rstr_1}. Equation \eqref{eq:rstr_2} is proved similarly using \eqref{eq:right_quant_hom__comult}.
  The second statement readily follows. 
\end{proof}


\begin{prop}\label{prop:fact}
  Suppose that $\bH$ is normal in $\bG$, and write $\widetilde{\pi} \in \Mor(\CzU{\widehat{\G}}, \CzU{\widehat{\bG/\bH}})$ for the strong quantum homomorphism associated to $\widehat{\bG/\bH}$ being a closed quantum subgroup of $\widehat{\G}$. Then the following conditions are equivalent:
  \begin{enumerate}
  \item \label{enu:fact__1} $U$ is trivial on $\bH$;
  \item \label{enu:fact__2}$U$ factors through $\bG/\bH$;
  \item \label{enu:fact__3}the representation $\Phi$ factors through $\widetilde{\pi}$.
  \end{enumerate}
\end{prop}
\begin{proof}
  The equivalence \prettyref{enu:fact__1}$\iff$\prettyref{enu:fact__2} is clear from the second assertion in \prettyref{lem:rstr} and $\li{\bG/\bH}=\li{\bH \backslash \bG}$ being the fixed-point algebra of $\a_l$.

  For the equivalence with \prettyref{enu:fact__3} we need the following observation: applying \eqref{eq:Woronowicz_Vaes} to $\widehat{\bG/\bH} \le \widehat{\G}$, in which case $\gamma$ is just the inclusion map $j : \Linfty{\bG/\bH} \hookrightarrow \Linfty{\G}$, gives that $j \circ \Lambda_{\bG/\bH} = \Lambda_{\G} \circ \widehat{\widetilde{\pi}}$.

  \prettyref{enu:fact__3}$\implies$\prettyref{enu:fact__2}: let $\Phi' \in \Mor(\CzU{\widehat{\bG/\bH}}, K(\H))$ be such that $\Phi = \Phi' \circ \widetilde{\pi}$. Then using \eqref{eq:dual_str_quant_hom},
\[
\begin{split}
U & = (\Lambda_\G \tensor \Phi)(\WW_\G) = (\Lambda_\G \tensor (\Phi' \circ \widetilde{\pi}))(\WW_\G) 
= ((\Lambda_\G \circ \widehat{\widetilde{\pi}})\tensor \Phi')(\WW_{\G / \HH}) \\
& = ((j \circ \Lambda_{\bG/\bH})\tensor \Phi')(\WW_{\G / \HH}) \in \Linfty{\G / \HH} \tensorn B(\H),
\end{split}
\]
proving that $U$ factors through $\bG/\bH$.

\prettyref{enu:fact__2}$\implies$\prettyref{enu:fact__3}: assume that $U$ factors through $\bG/\bH$; in other words, it can be seen as a representation of $\bG/\bH$. Let thus $\Phi' \in \Mor(\CzU{\widehat{\bG/\bH}}, K(\H))$ be the associated representation of $\CzU{\widehat{\bG/\bH}}$. Then repeating the above computation yields that $U = (\Lambda_\G \tensor (\Phi' \circ \widetilde{\pi}))(\WW_\G)$. The uniqueness of $\Phi$ hence implies that it equals $\Phi' \circ \widetilde{\pi}$.
\end{proof}

Next, we consider the (global) invariance of the space of $\bH$-invariant vectors under all of $\bG$ when the former is normal in the latter. 

\begin{prop}\label{prop:inv}
  Suppose that $\bH$ is normal in $\bG$. Then the closed subspace $\Inv (U|_\HH)$ of $\H$ consisting of all vectors invariant under $U|_\HH$ is preserved by $U$.
\end{prop}
\begin{proof}
  Let $P$ be the projection onto $\Inv (U|_\HH)$; its defining property is that it is the largest projection in $B(\H)$ such that
  \begin{equation*}
    U|_\HH(\one\otimes P)=\one\otimes P
  \end{equation*}
  (see \cite[Proposition 3.4]{Daws_Fima_Skalski_White_Haagerup_LCQG}). Our goal is to argue that $(\one \otimes P)U(\one \otimes P) = U(\one \otimes P)$.

  
  By \prettyref{lem:rstr} equation \eqref{eq:rstr_2} we have
  \begin{equation*}
    (\alpha_r\otimes\mathrm{id})(U(\one\otimes P)) = U_{13}(U|_\HH)_{23}(\one\otimes \one\otimes P) = U_{13}(\one\otimes \one\otimes P). 
  \end{equation*}
  This means that $U(\one\otimes P) \in \li{\bG/\bH} \tensorn B(H)$. Normality of $\bH$ in $\bG$ is equivalent to  $\li{\bG/\bH}=\li{\bH\backslash\bG}$. As a result, $(\alpha_l\otimes \mathrm{id})(U(\one\otimes P)) = U_{23}(\one\otimes \one\otimes P)$. In combination with \prettyref{lem:rstr} equation \eqref{eq:rstr_1}, we obtain
\begin{equation*}
  (U|_\HH)_{13}U_{23}(\one\otimes \one\otimes P) = U_{23}(\one\otimes \one\otimes P).
\end{equation*}
Equivalently, writing $U_0 := U(\one \otimes P)$, for all $\om \in \Lone{\G}$ we have
$$
U|_\HH \big( \one \tensor (\om \tensor \i)(U_0) \big)  = \one \tensor (\om \tensor \i)(U_0),
$$
hence $P (\om \tensor \i)(U_0) = (\om \tensor \i)(U_0)$ by the definition of $P$. That is, $(\one \otimes P)U(\one \otimes P) = U(\one \otimes P)$, as desired.
%
%
\end{proof}

We end this subsection with the following technical lemma, needed later.

\begin{lem}\label{lem:hom_sp_inv_norm_state_inv_vects} Suppose that the left action of $\bG$ on $\Linfty{\G/\HH}$ has an invariant normal state $\omega\in \li{\bG/\bH}_*$ (see \prettyref{def:invariant_weight} below), and extend it to a normal state  of $\li{\bG}$ denoted by the same symbol. Let $U$ be a representation of $\G$. Then for every $\xi\in\Inv(U|_{\HH})$ we have $(\om\tensor\i)(U)\xi\in\Inv(U)$.
\end{lem}
\begin{proof}
  From \prettyref{lem:rstr} and the assumption that $\xi\in\Inv(U|_{\HH})$
  we have
  \[
    ((\a_{r}\tensor\i)(U))(\Xi\tensor\xi)=U_{13}(U|_{\HH})_{23}(\Xi\tensor\xi)=U_{13}(\Xi\tensor\xi)\qquad(\forall\Xi\in\Ltwo{\G}\tensor\Ltwo{\HH}).
  \]
  Consequently, for each $\eta\in\H$ we have $\a_{r}\left((\i\tensor\om_{\xi,\eta})(U)\right)=(\i\tensor\om_{\xi,\eta})(U)\tensor\one$,
  i.e., $(\i\tensor\om_{\xi,\eta})(U)\in\Linfty{\G/\HH}$. 
  
  Since $U$ is a representation of $\G$, for all $\a,\be\in\Ltwo{\G}$
  and $\eta\in\H$ we have 
  \[
    \begin{split}\left\langle U\left[\a\tensor(\om\tensor\i)(U)\xi\right],\be\tensor\eta\right\rangle  & =\left\langle \left((\i\tensor\om\tensor\i)(U_{13}U_{23})\right)(\a\tensor\xi),\be\tensor\eta\right\rangle \\
      & =\left\langle \left((\i\tensor\om\tensor\i)((\Delta_\G\tensor\i)(U))\right)(\a\tensor\xi),\be\tensor\eta\right\rangle \\
      & =\left\langle ((\i\tensor\om)\circ\Delta_{\G})((\i\tensor\om_{\xi,\eta})(U))\a,\be\right\rangle .
    \end{split}
  \]
  Recall that $\Delta_{\G/\HH}:\Linfty{\G/\HH}\to\Linfty{\G}\tensorn\Linfty{\G/\HH}$
  is the restriction of $\Delta_{\G}$ to $\Linfty{\G/\HH}$, and that
  the invariance of $\om$ means that $(\i\tensor\om)\circ\Delta_{\G/\HH}=\om(\cdot)\one$.
  Hence, the above equals 
  \[
    \begin{split}\left\langle ((\i\tensor\om)\circ\Delta_{\G/\HH})((\i\tensor\om_{\xi,\eta})(U))\a,\be\right\rangle  & =\left\langle \om((\i\tensor\om_{\xi,\eta})(U))\a,\be\right\rangle \\
      & =\left\langle \a\tensor(\om\tensor\i)(U)\xi,\be\tensor\eta\right\rangle ,
    \end{split}
  \]
  proving that $(\om\tensor\i)(U)\xi\in\Inv(U)$.
\end{proof}

\subsection{Back to property (T)}

With the above material in place we can now mimic the proof of \cite[Proposition 4.13]{Bhattacharya_Brannan_Chirvasitu_Wang__rig_soft_prop_resid_fin_DQG}.

\begin{proof}[Proof of \prettyref{thm:ext}]
  ($\implies$): if $\G$ has property (T), then so does the pair $(\G,\HH)$; and furthermore, since $\widehat{\bG/\bH}$ is a closed quantum subgroup of $\widehat{\G}$ in the sense of (Vaes, thus) Woronowicz, \cite[Corollary 3.7]{Chen_Ng__prop_T_LCQGs} implies that $\bG/\bH$ has property (T) (use \eqref{eq:counit_str_quant_hom}). Remark that this is a particular case of \cite[Theorem 5.7]{Daws_Skalski_Viselter__prop_T}.
  
  ($\impliedby$): let $U$ be a representation of $\G$ on a Hilbert space $\H$ that has almost-invariant vectors. 
  \prettyref{prop:inv} then shows that $\Inv (U|_\HH)$ is preserved by $U$. Denoting the associated sub-representation of $U$ by $U_0$, we claim that it has almost-invariant vectors.

  To see this, note that any net $(\xi_{i})_{i \in \mathcal{I}}$ witnessing almost-invariant vectors of $U$ whose projections $(\xi^{\perp}_{i})_{i \in \mathcal{I}}$ on the orthogonal complement $\Inv (U|_\HH)^{\perp}$ fails to converge to zero would give rise to almost-invariant, and hence---because the pair $(\bG,\bH)$ has property (T)---non-zero $\bH$-invariant, vectors of $U$ in $\Inv (U|_\HH)^{\perp}$. This would then contradict the fact that $\Inv (U|_\HH)$ contains {\it all} such vectors. 

  \prettyref{prop:fact} shows that the representation $U_0$ of $\bG$ factors through $\bG/\bH$, and since the latter has property (T) the existence of almost-invariant vectors for $U_0$ entails the existence of a non-zero invariant vector for $U_0$ as a representation of $\bG/\bH$, thus also as a representation of $\bG$, concluding the proof.
\end{proof}

\section{Invariant weights}\label{sec:invariant-weights}

Classically, if $G$ is a locally compact group and $H$ is a closed subgroup of $G$, then the action $G\curvearrowright G/H$ admits a (strongly) \emph{quasi-invariant} (Radon) measure \cite[Proposition 2.54 and Theorem 2.56]{Folland__abs_harmon_anal}, but not always an \emph{invariant} measure. In fact, quasi-invariant measures on $G/H$ correspond to certain measures on $G$, namely the ones that are equivalent to the left/right Haar measure with the Radon--Nikodym derivative satisfying certain conditions \cite[Chapter VII, \S 2, Lemma 4, a$\iff$c, and Lemma 5]{Bourbaki__Integration_2_eng}. The existence of an invariant measure on $G/H$ is equivalent to the modular element of $G$ restricting to that of $H$ \cite[Theorem 2.49]{Folland__abs_harmon_anal}. In this section we prove that this holds for LCQGs.

Assume that $\HH$ is a closed quantum subgroup of a LCQG $\G$ in the sense of Woronowicz. As Kustermans remarks in \cite[p.~417]{Kustermans__induced_corep_LCQG}, every n.s.f.~weight on $\Linfty{\G/\HH}$ should be seen as playing the role of a quasi-invariant n.s.f.~weight (``measure''), because all n.s.f.~weights on $\Linfty{\G}$ are ``equivalent'' to one another---this is the essence of Connes' cocycle Radon--Nikodym derivative---and in particular to the Haar weights on $\G$ (compare \cite[Chapter VII, \S 2, Lemma 4, a$\iff$c]{Bourbaki__Integration_2_eng} again). To formalize the criterion for the existence of an invariant n.s.f.~weight for the action $\G\curvearrowright\Linfty{\G/\HH}$, recall that the Radon--Nikodym derivative \cite{Vaes__Radon_Nikodym} of the right Haar weight with respect to the left Haar weight of $\G$ is a generally unbounded, positive, self-adjoint operator $\delta_{\G}$ affiliated with $\li{\G}$, called the modular element. It has a universal version $\delta^{\mathrm{u}}_\G$ affiliated with the C$^*$-algebra $\CzU \G$.

\begin{defn}\label{def:invariant_weight}
  Let $\G$ be a LCQG, $\a$ be a left action of $\G$ on a von Neumann algebra $M$ and $\theta$ be a normal semi-finite weight on $M$. Consider the maps $(\i\tensor\theta)\circ\a$ and $\theta(\cdot)\one$, both from $M_{+}$ to $\epp{\Linfty{\G}}$. We say that $\theta$ is \emph{completely invariant} under $\a$ if these maps coincide. We say that $\theta$ is\emph{ invariant} under $\a$ if they coincide on $\mathcal{M}_{\theta}^{+}$; equivalently, if for every $x\in\mathcal{M}_{\theta}^{+}$ and $\omega\in\Lone{\G}_{+}$ we have $\theta\left((\omega\tensor\i)(\a(x))\right)=\theta(x)\omega(\one)$ (in particular, $(\omega\tensor\i)(\a(x)) \in \mathcal{M}_{\theta}^{+}$). These invariance notions are defined similarly for right actions.
\end{defn}

\begin{rem}
  \label{rem:inv}
  The definition of invariance (of unbounded weights) is a particular case of \citep[Definition 2.3]{Vaes__unit_impl_LCQG}.  Complete invariance evidently implies invariance. In certain cases these notions coincide, e.g.~for the Haar weights of a LCQG, see \citep[Proposition 3.1]{Kustermans_Vaes__LCQG_von_Neumann}.
\end{rem}

\begin{defn}
Let $\HH$ be a closed quantum subgroup of a LCQG $\G$ in the sense
of Woronowicz, and write $\pi:\CzU{\G}\to\CzU{\HH}$ for the corresponding
strong quantum homomorphism. We say that $\delta_{\G}$ \emph{restricts} \emph{to} $\delta_{\HH}$
if $\pi((\delta_{\G}^{\mathrm{u}})^{it})=(\delta_{\HH}^{\mathrm{u}})^{it}$
for all $t\in\R$.
\end{defn}

The next result, which is the main one of this section, extends the above-mentioned classical result, as well as  \citep[Proposition 5.1]{Kalantar_Kasprzak_Skalski__open_quant_subgr} and \citep[Lemma 3.1]{Kalantar_Kasprzak_Skalski_Soltan__ind_LCQG_revisited}.

\begin{thm}
  \label{thm:quot_inv_nsf_weight}
  Let $\HH$ be a closed quantum subgroup of a LCQG $\G$ in the sense of Vaes and
  \begin{equation}\label{eq:ar-action}
    \alpha_r:\li{\bG}\to \li{\bG}\tensorn \li{\bH} 
  \end{equation}
  the resulting right action of $\HH$ on $\Linfty{\G}$. Then the following conditions are equivalent:
  \begin{enumerate}
  \item\label{enu:inv1} the left action of $\G$ on $\Linfty{\G/\HH}$ has a completely invariant n.s.f.~weight;
  \item\label{enu:inv2} $\a_r(\delta_{\G}^{it})=\delta_{\G}^{it}\tensor\delta_{\HH}^{it}$ for all $t\in\R$;
  \item\label{enu:inv3} $\delta_{\G}$ restricts to $\delta_{\HH}$.   
  \end{enumerate}
  In that case, the invariant n.s.f.~weight is unique up to scaling. 
\end{thm}

We make note of the following consequence.

\begin{cor}\label{cor:g-so-h-unim}
  Let $\HH$ be a closed quantum subgroup of a LCQG $\G$ in the sense of Vaes, and suppose the left action of $\G$ on $\Linfty{\G/\HH}$ has a completely invariant n.s.f.~weight. If $\bG$ is unimodular, then so is $\bH$.
\end{cor}
\begin{proof}
  This is immediate from \prettyref{thm:quot_inv_nsf_weight}: given condition \prettyref{enu:inv2}, $\delta_{\bG}^{it}=1$ for all $t$ implies the same for $\delta_{\bH}$.
\end{proof}

Suppose that $\HH$ is a closed quantum subgroup of a LCQG $\G$ in the sense of Vaes. By \citep[Theorem 5.2]{Kasprzak_Khosravi_Soltan__int_act_QGs} (attributed there to \citep[Proposition 3.12]{De_Commer__Galois_obj_cocy_tw_LCQG}), the right action $\a_r$ of $\HH$ on $\Linfty{\G}$ (see \prettyref{sec:prelim}) is \emph{integrable} in the sense of \citep[\S 6]{Kustermans__induced_corep_LCQG}; see also \citep[Corollary 5.6]{Kasprzak_Khosravi_Soltan__int_act_QGs}.  Thus, the function $\mathcal{T}:=(\i\tensor\varphi_{\HH})\circ\a_r:\Linfty{\G}_{+}\to\epp{\Linfty{\G}}$ is an n.s.f.~operator-valued weight from $\Linfty{\G}$ to $\Linfty{\G/\HH}$ under the canonical embedding of $\epp{\Linfty{\G/\HH}}$ inside $\epp{\Linfty{\G}}$ (\citep[Proposition 1.3]{Vaes__unit_impl_LCQG}, noting the different convention in the definition of a (co-) representation; see \citep[\S 8, p.~452]{Kustermans__induced_corep_LCQG}).  Therefore, each n.s.f.~weight $\theta$ on $\Linfty{\G/\HH}$ induces the n.s.f.~weight $\theta\circ\mathcal{T}$ on $\Linfty{\G}$. These weights are characterized by the next result, which generalizes the classical correspondence alluded to above between quasi-invariant measures on $G/H$ and certain measures on $G$.

\begin{thm}[{\citep[Propositions 8.6 and 8.7]{Kustermans__induced_corep_LCQG}}]
  \label{thm:induced_weights_from_quotients} Let $\HH$ be a closed quantum subgroup of a LCQG $\G$ in the sense of Vaes. Use the notation of the previous paragraph. An n.s.f.~weight $\phi$ on $\Linfty{\G}$ has the form $\theta\circ\mathcal{T}$ for some n.s.f.~weight $\theta$ on $\Linfty{\G/\HH}$ if and only if for all $t\in\R$,
  \[
    \a_r({(D\phi:D\psi_{\G})}_{t})={(D\phi:D\psi_{\G})}_{t}\tensor\delta_{\HH}^{-it},
  \]
  where $(D\phi:D\psi_{\G})$ is Connes' cocycle derivative of $\phi$ with respect to $\psi_{\G}$. In that case, $\theta$ is unique. 
\end{thm}

In the proof of \prettyref{thm:quot_inv_nsf_weight} we will use a few standard manipulations of operator-valued weights, such as extending them (normally) to the extended positive parts, composing and tensoring them, etc. Note that just like normal operator-valued weights, a positive normal linear map $S$ from a von Neumann algebra $M$ to a von Neumann algebra $N$ extends uniquely to a map $S:\widehat{M}_{+}\to\widehat{N}_{+}$ that is normal in the sense that if $\left(m_{i}\right)$ is an increasing net in $\widehat{M}_{+}$ that converges (pointwise, on $M_{*}^{+}$) to $m\in\widehat{M}_{+}$, then the increasing net $\left(Sm_{i}\right)$ converges to $Sm$.

\begin{lem}
  \label{lem:quot_inv_nsf_weight}
  In the setting of \prettyref{thm:quot_inv_nsf_weight}, denote by $\mathcal{T}$ the canonical operator-valued weight from $\Linfty{\G}$ to $\Linfty{\G/\HH}$. Recall that $\Delta_{\G/\HH}$ stands for the left action  of $\G$ on $\Linfty{\G/\HH}$.
  \begin{enumerate}
  	\item \label{enu:quot_inv_nsf_weight__1} For every $\om \in \Lone{\G}_+$ we have $(\om\tensor\i)\circ\Delta_{\G/\HH}\circ\mathcal{T}=\mathcal{T}\circ(\om\tensor\i)\circ\Delta_{\G}$, where in the left hand side we  extend $(\om\tensor\i)\circ\Delta_{\G/\HH}$ to $\epp{\Linfty{\G/\HH}}$.
  	\item \label{enu:quot_inv_nsf_weight__2} An n.s.f.~weight $\theta$ on $\Linfty{\G/\HH}$ is completely invariant under  $\Delta_{\G/\HH}$ if and only if $\theta\circ\mathcal{T}$ is completely invariant under $\Delta_\G$, if and only if $\theta\circ\mathcal{T}$ equals $\varphi_{\G}$ up to scaling by a positive scalar.
  \end{enumerate}
\end{lem}
\begin{proof} In what follows we tacitly extend maps to the extended positive parts of the respective von Neumann algebras as required for the statements to make sense.

\ref{enu:quot_inv_nsf_weight__1} Writing again $j$ for the inclusion map of $\Linfty{\G/\HH}$ in $\Linfty{\G}$ and recalling that $j\circ\mathcal{T}=(\i\tensor\varphi_{\HH})\circ\a_r$, for all $\om \in \Lone{\G}_+$ we have
\[
\begin{split}j\circ(\om\tensor\i)\circ\Delta_{\G/\HH}\circ\mathcal{T} & =(\om\tensor\i)\circ\Delta_{\G}\circ j\circ\mathcal{T}=(\om\tensor\i)\circ\Delta_{\G}\circ(\i\tensor\varphi_{\HH})\circ\a_r\\
& =(\om\tensor\i)\circ(\i\tensor\i\tensor\varphi_{\HH})\circ(\Delta_{\G}\tensor\i)\circ\a_r\\
& =(\om\tensor\i)\circ(\i\tensor\i\tensor\varphi_{\HH})\circ(\i\tensor\a_r)\circ\Delta_{\G}\\
& =(\i\tensor\varphi_{\HH})\circ\a_r\circ(\om\tensor\i)\circ\Delta_{\G}=j\circ\mathcal{T}\circ(\om\tensor\i)\circ\Delta_{\G}
\end{split}
\]
(the reader can easily justify equalities like $\Delta_{\G}\circ(\i\tensor\varphi_{\HH})=(\i\tensor\i\tensor\varphi_{\HH})\circ(\Delta_{\G}\tensor\i)$ as maps from $\epp{\left(\Linfty{\G}\tensorn\Linfty{\HH}\right)}$ to $\epp{\left(\Linfty{\G}\tensorn\Linfty{\G}\right)}$). Thus, $(\om\tensor\i)\circ\Delta_{\G/\HH}\circ\mathcal{T}=\mathcal{T}\circ(\om\tensor\i)\circ\Delta_{\G}$.
	
  \ref{enu:quot_inv_nsf_weight__2} Complete invariance of $\theta$ under $\Delta_{\G/\HH}$ means that $(\i\tensor\theta)\circ\Delta_{\G/\HH}=\theta(\cdot)\one$ as maps
  \begin{equation*}
    \Linfty{\G/\HH}_{+}\to\epp{\Linfty{\G}}
  \end{equation*}
  or equivalently as maps $\epp{\Linfty{\G/\HH}}\to\epp{\Linfty{\G}}$. Since $\mathcal{T}$ maps $\epp{\Linfty{\G}}$ onto $\epp{\Linfty{\G/\HH}}$ \citep[Proposition 2.5]{Haagerup__oper_val_weights_1}, that is equivalent to the equality
\[
(\i\tensor\theta)\circ\Delta_{\G/\HH}\circ\mathcal{T}=(\theta\circ\mathcal{T})(\cdot)\one,
\]
which is the same as
\begin{equation*}
  \theta\circ(\om\tensor\i)\circ\Delta_{\G/\HH}\circ\mathcal{T}=\om(\one)\theta\circ\mathcal{T}\qquad(\forall\om\in\Lone{\G}_{+}).
\end{equation*}
By \ref{enu:quot_inv_nsf_weight__1}, this is equivalent to 
\begin{equation*}
\theta\circ\mathcal{T}\circ(\om\tensor\i)\circ\Delta_{\G}=\om(\one)\theta\circ\mathcal{T}\qquad(\forall\om\in\Lone{\G}_{+}),
\end{equation*}
meaning that $\theta\circ\mathcal{T}$ is completely invariant under $\Delta_\G$.
From the uniqueness of the left Haar weight of $\G$ (and \prettyref{rem:inv}), this is equivalent to $\theta\circ\mathcal{T}$ being equal to $\varphi_{\G}$ up to scaling by a positive scalar.
\end{proof}

\begin{proof}[Proof of \prettyref{thm:quot_inv_nsf_weight}]
  The up-to-scaling uniqueness is already part of \prettyref{thm:induced_weights_from_quotients}.
  
  By \prettyref{lem:quot_inv_nsf_weight} \ref{enu:quot_inv_nsf_weight__2}, an n.s.f.~weight $\theta$ on $\Linfty{\G/\HH}$ is completely invariant under the left action $\Delta_{\G/\HH}$ of $\G$ on $\Linfty{\G/\HH}$ if and only if $\theta\circ\mathcal{T}$ equals $\varphi_{\G}$ up to scaling. By \prettyref{thm:induced_weights_from_quotients}, such $\theta$ exists if and only if
\[
\a_r({(D\varphi_{\G}:D\psi_{\G})}_{t})={(D\varphi_{\G}:D\psi_{\G})}_{t}\tensor\delta_{\HH}^{-it}\qquad(\forall t\in\R),
\]
and since ${(D\psi_{\G}:D\varphi_{\G})}_{t}=\nu^{\frac{1}{2}it^{2}}\delta_{\G}^{it}$ for all $t\in\R$,
this is equivalent to 
\begin{equation}
\a_r(\delta_{\G}^{it})=\delta_{\G}^{it}\tensor\delta_{\HH}^{it}\qquad(\forall t\in\R).\label{eq:quot_inv_nsf_weight}
\end{equation}
This proves the equivalence of the first two conditions in \prettyref{thm:quot_inv_nsf_weight}. 

Denoting by $\pi:\CzU{\G}\to\CzU{\HH}$ the strong quantum homomorphism that corresponds to $\HH$ being a closed quantum subgroup of $\G$, we have $\a_r|_{\Cz{\G}}\circ\Lambda_{\G}=(\Lambda_{\G}\tensor\Lambda_{\HH}\pi)\circ\Delta_{\G}^{\mathrm{u}}$ as morphisms from $\CzU{\G}$ to $\Cz{\G}\tensormin\Cz{\HH}$ by \eqref{eq:right_quant_hom__comult}. Since for all $t\in\R$ we have $\left(\delta_{\G}^{\mathrm{u}}\right)^{it}\in\M{\CzU{\G}}$, $\Delta_{\G}^{\mathrm{u}}(\left(\delta_{\G}^{\mathrm{u}}\right)^{it})=\left(\delta_{\G}^{\mathrm{u}}\right)^{it}\tensor\left(\delta_{\G}^{\mathrm{u}}\right)^{it}$ and $\Lambda_{\G}(\left(\delta_{\G}^{\mathrm{u}}\right)^{it})=\delta_{\G}^{it}$, condition \prettyref{eq:quot_inv_nsf_weight} is equivalent to
\[
\delta_{\G}^{it}\tensor\Lambda_{\HH}\bigl(\pi(\left(\delta_{\G}^{\mathrm{u}}\right)^{it})\bigr)=\delta_{\G}^{it}\tensor\delta_{\HH}^{it}\qquad(\forall t\in\R),
\]
that is,
\[
\Lambda_{\HH}\bigl(\pi(\left(\delta_{\G}^{\mathrm{u}}\right)^{it})\bigr)=\delta_{\HH}^{it}\qquad(\forall t\in\R).
\]
That this is equivalent to $\pi(\left(\delta_{\G}^{\mathrm{u}}\right)^{it})=\left(\delta_{\HH}^{\mathrm{u}}\right)^{it}$ for all $t\in\R$ (i.e.~to $\delta_{\G}$ restricting to $\delta_{\HH}$) 
follows from \citep[Result 6.1]{Kustermans__LCQG_universal}, because for every $t \in \R$, both $\pi(\left(\delta_{\G}^{\mathrm{u}}\right)^{it})$ and $\left(\delta_{\HH}^{\mathrm{u}}\right)^{it}$ are group-like unitaries in $\M{\CzU{\bH}}$ (that is, $1$-dimensional representations of $\HH$ in the universal sense), so they are equal if and only if applying $\Lambda_{\HH}$ to them yields the same element, namely $\delta_{\HH}^{it}$ (also compare the proof of \citep[Proposition 10.1]{Kustermans__LCQG_universal}).
\end{proof}

\begin{rem}\label{rem:hom_sp_inv_weight_faithful}
  If $\HH$ is a closed quantum subgroup of a LCQG $\G$ in the sense of Woronowicz, then every non-zero normal semi-finite weight $\theta$ on $\Linfty{\G/\HH}$ that is completely invariant under the left action of $\G$ is necessarily faithful. Indeed, if $0 \neq x \in \Linfty{\G/\HH}_+$, apply the right Haar weight $\psi_{\G}$ to the equality $(\i\tensor\theta)(\Delta_{\G / \bH}(x)) = \theta(x)\one$. The left hand side equals $\psi_{\G}(x) \theta(\one)$ because $\Delta_{\G / \bH}$ is the restriction of $\Delta_{\G}$ (use the complete right invariance of $\psi_{\G}$). Since $\psi_{\G}$ is faithful and $\theta$ is non-zero, we must have $\theta(x) > 0$.
  
  We will henceforth use this implicitly without further comment. 
\end{rem}

It is now a simple remark that the conditions in \prettyref{thm:quot_inv_nsf_weight} hold for {\it normal} $\bH\trianglelefteq \bG$.

\begin{cor}\label{cor:norm-com-inv}
  If $\bH\trianglelefteq\bG$ is a normal closed quantum subgroup then the left action of $\G$ on $\Linfty{\G/\HH}$ has a completely invariant n.s.f.~weight, namely the left Haar weight $\varphi_{\bG/\bH}$ of $\bG/\bH$.
\end{cor}
\begin{proof}
  This follows from \prettyref{lem:quot_inv_nsf_weight} \ref{enu:quot_inv_nsf_weight__2} and, with $\cT$ as in that lemma, the fact that for normal subgroups the Weyl-type ``disintegration'' formula
  \begin{equation*}
    \varphi_{\bG/\bH}\circ \cT = \varphi_{\bG}
  \end{equation*}
  holds (up to scaling) by \cite[Proposition 4.10]{chk}.
\end{proof}

In particular, \prettyref{cor:g-so-h-unim} implies:

\begin{cor}\label{cor:norm-unim}
  Normal closed quantum subgroups of unimodular LCQGs are again unimodular.
\end{cor}

\section{The canonical implementation and invariant normal states}\label{sec:canon-impl-invar}

The next result extends \citep[Proposition 4.11 (b)]{Daws_Skalski_Viselter__prop_T} from discrete to locally compact quantum groups. Our proof strategy is very different.
\begin{thm}\label{thm:inv-can}
  \label{thm:inv_norm_pos_func}Let $\a$ be an action of a LCQG $\G$ on a von Neumann algebra $N$. Let $\rho\in N_{*}^{+}$.  Then $\rho$ is invariant under $\a$ if and only if the unique vector $\z$ in the positive cone $\Ltwo N_{+}$ such that $\rho=\om_{\z}$ is invariant under the canonical unitary implementation \cite{Vaes__unit_impl_LCQG} of $\a$.
\end{thm}
In the following proofs we will work with left actions for convenience.
\begin{lem}\label{lem:act_fixed_point}
	Let $\a$ be an action of a LCQG $\G$ on a von Neumann algebra $N$.
	If $p\in N$ is a projection such that either $\a(p)\ge\one\tensor p$
	or $\a(p)\le\one\tensor p$, then $p$ is a fixed point of $\a$, that is,
	$\a(p)=\one\tensor p$.
\end{lem}
\begin{proof}
	This is proved in \citep[Lemma 3.1 and Remark 3.2]{Kasprzak_Khosravi_Soltan__int_act_QGs}
	in the case that $\a$ is \emph{ergodic}. The proof in the general
	case is just the same, removing all references to ergodicity and changing
	the proof's last line appropriately, as we now explain. For convenience,
	we use the convention of \citep{Kasprzak_Khosravi_Soltan__int_act_QGs},
	so $\a$ is a right action and $\a(p)\le p\tensor\one$. Write $q:=\a(p)\in N\tensorn\Linfty{\G}$.
	The proof of \citep[Lemma 3.1]{Kasprzak_Khosravi_Soltan__int_act_QGs}
	shows, without using ergodicity, that $(\i\tensor\Delta)(q)=q\tensor\one$ (see p.~3234 line 4 therein).
	So by the defining property of right actions, 
	\[
	(\a\tensor\i)(\a(p))=(\i\tensor\Delta)(\a(p))=\a(p)\tensor\one=(\a\tensor\i)(p\tensor\one).
	\]
	Injectivity of $\a\tensor\i$ now implies that $\a(p)=p\tensor\one$.
\end{proof}

\begin{lem}
  \label{lem:inv_norm_pos_func_supp}Let $\a$ be an action of a LCQG $\G$ on a von Neumann algebra $N$. If $\rho\in N_{*}^{+}$ is invariant under $\a$, then $p:=\supp(\rho)$ is a fixed point of $\a$.
\end{lem}
\begin{proof}
  Note that if $a,e$ are elements of a C$^{*}$-algebra with $0\le a$, $\left\Vert a\right\Vert \le1$ and $e$ being a projection, then $eae=e$ if and only if $a\ge e$, because working in some unitization, $eae=e$ $\iff$ $e(\one-a)e=0$ $\iff$ $\left(\one-a\right)^{1/2}e=0$ $\iff$ $\left(\one-a\right)e=0$ $\iff$ $a\ge a^{1/2}ea^{1/2}=e$ (using the commutation of $a,e$).

We can assume that $\rho$ is a state. From the $\a$-invariance of
$\rho$ we get 
\begin{equation}
(\i\tensor\rho)(\a(p))=\rho(p)\one=(\i\tensor\rho)(\one\tensor p).\label{eq:inv_norm_pos_func_supp}
\end{equation}
 If the positive element $\one\tensor p-(\one\tensor p)\a(p)(\one\tensor p)$
were non-zero, there would exist $\om\in N_{*}^{+}$ such that $(\om\tensor\i)\left(\one\tensor p-(\one\tensor p)\a(p)(\one\tensor p)\right)$
is a non-zero element of $pNp$, thus $(\om\tensor\rho)\left(\one\tensor p-(\one\tensor p)\a(p)(\one\tensor p)\right)\neq0$,
contradicting \eqref{eq:inv_norm_pos_func_supp}. Thus $(\one\tensor p)\a(p)(\one\tensor p)=\one\tensor p$,
so that $\a(p)\ge\one\tensor p$ by the preceding paragraph. This
entails that $\a(p)=\one\tensor p$ by \prettyref{lem:act_fixed_point}.
\end{proof}
The following lemma summarizes a few well-known facts from modular
theory. 
\begin{lem}
\label{lem:red_weight_GNS}Let $N$ be a von Neumann algebra, $\theta$
be an n.s.f.~weight on $N$ and $p$ be a projection in the centralizer
of $\theta$. Then on the reduced von Neumann algebra $pNp$ we have
the reduced (n.s.f.)~weight $\theta_{p}:=\theta|_{pNp}$. We have
$p\mathcal{N}_{\theta}p=\mathcal{N}_{\theta}\cap pNp=\mathcal{N}_{\theta_{p}}$,
the GNS Hilbert space $\Ltwo{pNp,\theta_{p}}$ naturally identifies
with the subspace $\overline{\gnsmap_{\theta}(p\mathcal{N}_{\theta}p)}$
of $\Ltwo{N,\theta}$, and upon making this identification, the restriction
$\gnsmap_{\theta}|_{p\mathcal{N}_{\theta}p}$ equals $\gnsmap_{\theta_{p}}$
. Furthermore, the subspace $\overline{\gnsmap_{\theta}(p\mathcal{N}_{\theta}p)}$
of $\Ltwo{N,\theta}$ is reducing for $T_{\theta}$, and the part
of $T_{\theta}$ on this subspace is precisely $T_{\theta_{p}}.$
\end{lem}
\begin{proof}
	The reduced weight $\theta_{p}$ is clearly normal and faithful and $\mathcal{N}_{\theta}\cap pNp=\mathcal{N}_{\theta_{p}}$.
	Recall that $\mathcal{N}_{\theta}$ and $\mathcal{M}_{\theta}$ are
	bimodules over the algebra of $\sigma^{\theta}$-entire analytic elements.
	Since $p$ is in the centralizer of $\theta$, it is entire analytic,
	hence $\theta_{p}$ is semi-finite and $p\mathcal{N}_{\theta}p=\mathcal{N}_{\theta}\cap pNp=\mathcal{N}_{\theta_{p}}$.
	Furthermore, $pNp$ is invariant under $\sigma^{\theta}$, and the
	restriction of $\sigma^{\theta}$ to $pNp$ is precisely $\sigma^{\theta_{p}}$
	by the modular automorphism group uniqueness theorem, because $\theta_{p}$
	is $\sigma^{\theta}|_{pNp}$-invariant and $\theta_{p}$ satisfies
	the KMS-condition with respect to $\sigma^{\theta}|_{pNp}$ (this
	argument is given in \citep[Proof of Lemme 3.2.6]{Connes__classification_des_facteurs}
	for when $N$ is a factor, but this condition is not required).
	
	From this point one proceeds like in the relevant part of the proof
	of Takesaki's conditional expectation theorem (see, e.g., \citep[10.2, pp.~130--131]{Stratila__mod_thy}).
	The closed subspace $\overline{\gnsmap_{\theta}(p\mathcal{N}_{\theta}p)}$
	of $L^{2}(N,\theta)$ is invariant under $pNp$, and $(\overline{\gnsmap_{\theta}(p\mathcal{N}_{\theta}p)},\gnsmap_{\theta}|_{p\mathcal{N}_{\theta}p})$
	together with the identity representation is clearly a GNS construction
	for $(pNp,\theta_{p})$. Thus we can and will identify $L^{2}(pNp,\theta_{p})$
	with $\overline{\gnsmap_{\theta}(p\mathcal{N}_{\theta}p)}$ and $\gnsmap_{\theta_{p}}$
	with $\gnsmap_{\theta}|_{p\mathcal{N}_{\theta}p}$. For all $t\in\R$
	and $x\in\mathcal{N}_{\theta_{p}}=p\mathcal{N}_{\theta}p$,
	\[
	\nabla_{\theta_{p}}^{it}\gnsmap_{\theta_{p}}(x)=\gnsmap_{\theta_{p}}(\sigma_{t}^{\theta_{p}}(x))=\gnsmap_{\theta}(\sigma_{t}^{\theta}(x))=\nabla_{\theta}^{it}\gnsmap_{\theta}(x)=\nabla_{\theta}^{it}\gnsmap_{\theta_{p}}(x).
	\]
	That is to say, for each $t\in\R$, $\overline{\gnsmap_{\theta}(p\mathcal{N}_{\theta}p)}$
	is reducing for $\nabla_{\theta}^{it}$ and the restriction equals
	$\nabla_{\theta_{p}}^{it}$. Equivalently, $\overline{\gnsmap_{\theta}(p\mathcal{N}_{\theta}p)}$
	is reducing for $\nabla_{\theta}^{1/2}$ and the restriction equals
	$\nabla_{\theta_{p}}^{1/2}$. Now, for $x\in\mathcal{N}_{\theta_{p}}\cap\mathcal{N}_{\theta_{p}}^{*}$
	\[
	T_{\theta_{p}}\gnsmap_{\theta_{p}}(x)=\gnsmap_{\theta_{p}}(x^{*})=\gnsmap_{\theta}(x^{*})=T_{\theta}\gnsmap_{\theta}(x),
	\]
	that is,
	\[
	J_{\theta_{p}}\nabla_{\theta_{p}}^{1/2}\gnsmap_{\theta_{p}}(x)=J_{\theta}\nabla_{\theta}^{1/2}\gnsmap_{\theta}(x)=J_{\theta}\nabla_{\theta_{p}}^{1/2}\gnsmap_{\theta_{p}}(x).
	\]
	So $\overline{\gnsmap_{\theta}(p\mathcal{N}_{\theta}p)}$ is also
	reducing for $J_{\theta}$ and the restriction equals $J_{\theta_{p}}$.
	All in all, $\overline{\gnsmap_{\theta}(p\mathcal{N}_{\theta}p)}$
	is reducing for $T_{\theta}$ and the restriction equals $T_{\theta_{p}}$.
\end{proof}

\begin{proof}[Proof of \prettyref{thm:inv_norm_pos_func}]
Denote the canonical unitary implementation of $\a$ by $U_{\a}$.
Sufficiency is clear: if $\z\in\Ltwo N$ is invariant under $U_{\a}$,
then $\om_{\z}\in N_{*}$ is invariant under $\a$.

Necessity: assume that a state $\rho\in N_{*}$ is invariant under
$\a$. Write $p$ for the support of $\rho$, let $\rho'$ be a normal
semi-finite weight on $N$ whose support is $\one-p$, and set $\theta:=\rho+\rho'$.
Then $\theta$ is an n.s.f.~weight on $N$. We should prove that
the unit vector $\gnsmap_{\theta}(p)$, which is the (unique) element
$\z$ of $\Ltwo{N,\theta}_{+}$ such that $\rho=\om_{\z}$, is invariant
under $U_{\a}$. The idea is to reduce the problem to the case $p=\one$.

Let $\widetilde{\theta}$ be the n.s.f.~weight on $\G\pres{}{\a}{\ltimes}N$
that is dual to $\theta$ \citep[Definition 3.1]{Vaes__unit_impl_LCQG}.
We use its GNS construction afforded by \citep[Definition 3.4 and Proposition 3.10]{Vaes__unit_impl_LCQG},
so the GNS Hilbert space into which $\gnsmap_{\widetilde{\theta}}$
maps is $\Ltwo{\G}\tensor\Ltwo{N,\theta}$, and we have 
\begin{equation}
\left\{ (a\tensor\one)\a(x):a\in\mathcal{N}_{\widehat{\varphi}},x\in\mathcal{N}_{\theta}\right\} \text{ is a }\text{\ensuremath{*}-ultrastrong\textendash norm core of }\gnsmap_{\widetilde{\theta}}\label{eq:inv_norm_pos_func__gns_theta_tilde__1}
\end{equation}
and
\begin{equation}
\gnsmap_{\widetilde{\theta}}\left((a\tensor\one)\a(x)\right)=\gnsmap_{\widehat{\varphi}}(a)\tensor\gnsmap_{\theta}(x)\qquad(\forall a\in\mathcal{N}_{\widehat{\varphi}},x\in\mathcal{N}_{\theta}).\label{eq:inv_norm_pos_func__gns_theta_tilde__2}
\end{equation}
Note that $\one\otimes p=\alpha(p)\in\G\pres{}{\a}{\ltimes}N$ belongs to the centralizer of $\widetilde{\theta}$, because $p$ belongs to the centralizer of $\theta$, i.e., $\left(\sigma_{t}^{\theta}\right)_{t\in\R}$ fixes $p$, and $\sigma_{t}^{\widetilde{\theta}}\circ\alpha=\alpha\circ\sigma_{t}^{\theta}$ for all $t\in\R$ by \citep[Proposition 3.7]{Vaes__unit_impl_LCQG}.

We will use the description of $U_{\a}$ as $(J_{\widehat{\varphi}}\otimes J_{\theta})J_{\widetilde{\theta}}$ (recall that up to unitary equivalence, $U_{\a}$ does not depend on the chosen n.s.f.~weight \citep[Proposition 4.1]{Vaes__unit_impl_LCQG}).  We have to show that
\begin{equation}
  U_{\a}(\xi\otimes\gnsmap_{\theta}(p))=\xi\otimes\gnsmap_{\theta}(p)\qquad(\forall\xi\in L^{2}(\G)).\label{eq:inv_norm_pos_func__canon_unit_impl}
\end{equation}

Since $p$ belongs to the centralizer of $\theta$ and $q:=\alpha(p)=\one\otimes p$ belongs to the centralizer of $\widetilde{\theta}$, we can apply \prettyref{lem:red_weight_GNS} to these cases. Note that the reduced weight $\theta_{p}$ is the faithful normal state $\rho_{p}$. Denoting by $\a_{p}:=\a|_{pNp}:pNp\to\Linfty{\G}\tensorn pNp$ the reduced action of $\G$ on $pNp$, the reader can check using \citep[Lemma 3.3]{Vaes__unit_impl_LCQG} that $q(\G\pres{}{\a}{\ltimes}N)q=\G\pres{}{\a_{p}}{\ltimes}pNp$ and $(\widetilde{\theta})_{q}=\widetilde{(\theta_{p})}=\widetilde{(\rho_{p})}$, where $\widetilde{(\rho_{p})}$ is the dual weight of $\rho_{p}$ constructed from $\a_{p}$ and $\rho_{p}$ like $\widetilde{\theta}$ was constructed from $\a$ and $\theta$. Recall that we view $\G\pres{}{\a}{\ltimes}N$ as acting standardly on $\Ltwo{\G}\tensor\Ltwo{N,\theta}$ identified with $\Ltwo{\G\pres{}{\a}{\ltimes}N,\widetilde{\theta}}$ by \eqref{eq:inv_norm_pos_func__gns_theta_tilde__1} and \eqref{eq:inv_norm_pos_func__gns_theta_tilde__2}. Observe that $\H:=\overline{\gnsmap_{\widetilde{\theta}}(q\mathcal{N}_{\widetilde{\theta}}q)}$ equals $\Ltwo{\G}\tensor\overline{\gnsmap_{\theta}(pNp)}$ by \eqref{eq:inv_norm_pos_func__gns_theta_tilde__1} and \eqref{eq:inv_norm_pos_func__gns_theta_tilde__2} because $p,q$ belong to the suitable centralizers.  Finally, note that the two natural ways of viewing $q(\G\pres{}{\a}{\ltimes}N)q$ as acting standardly on $\H$ and the corresponding GNS maps agree.

We claim that $\H$ is a reducing subspace for $U_{\a}$, and that the restriction is precisely the canonical implementing unitary $U_{\a_{p}}$ of $\a_{p}$ constructed from $\rho_{p}$. Indeed, by \prettyref{lem:red_weight_GNS}, $\overline{\gnsmap_{\theta}(pNp)}$ is reducing for $J_{\theta}$ and the restriction is $J_{\theta_{p}}=J_{\rho_{p}}$, and similarly, $\H$ is reducing for $J_{\widetilde{\theta}}$ and the restriction is $J_{(\widetilde{\theta})_{q}}=J_{\widetilde{(\rho_{p})}}$. So all in all, $\H$ is reducing for $(J_{\widehat{\varphi}}\otimes J_{\theta})J_{\widetilde{\theta}}=U_{\a}$, and the restriction is $(J_{\widehat{\varphi}}\otimes J_{\rho_{p}})J_{\widetilde{(\rho_{p})}}=U_{\a_{p}}$.

In conclusion, by passing to $q(\G\pres{}{\a}{\ltimes}N)q$ we transport the verification of the claim to the case $p=\one$, so we can assume that $\theta$ is a faithful normal $\a$-invariant \emph{state}.  But in this case, the canonical unitary implementation $U_{\a}$ is given by the simple formula
\[
(\om \tensor \i)(U_{\a}^*) \gnsmap_\theta(x) = \gnsmap_\theta( (\om \tensor \i)(\a(x)) )\qquad (\forall x\in N, \om \in \Lone{\G}),
\]
see the paragraph preceding \cite[Definition 4.4]{Daws_Skalski_Viselter__prop_T}. Hence, $\gnsmap_{\theta}(\one)$ is invariant under $U_{\a}$.
\end{proof}

In the rest of this section we present several applications of \prettyref{thm:inv_norm_pos_func}.
\begin{defn}
  Let $\HH$ be a closed quantum subgroup of a LCQG $\G$ in the sense of Woronowicz. The canonical unitary implementation of the left action of $\G$ on $\Linfty{\G/\HH}$ will be called the \emph{quasi-regular representation} of $\G/\HH$.
\end{defn}
We have the following immediate consequence of \prettyref{thm:inv_norm_pos_func}. Its classical version is a consequence of \citep[Theorem E.3.1]{Bekka_de_la_Harpe_Valette__book} (take $\sigma$ to be the trivial representation there), as the induction of the trivial representation gives the quasi-regular representation \cite[Example E.1.8 (ii)]{Bekka_de_la_Harpe_Valette__book}.

\begin{thm}
  Let $\HH$ be a closed quantum subgroup of a LCQG $\G$ in the sense of Woronowicz. The left action of $\G$ on $\Linfty{\G/\HH}$ has an invariant normal state if and only if the quasi-regular representation of $\G/\HH$ has a non-zero invariant vector.
\end{thm}

\subsection{An application to amenability of representations and a related notion}
\begin{defn}[\citep{Bedos_Conti_Tuset__amen_co_amen_alg_QGs_coreps,Bedos_Tuset_2003,Ng__amen_rep_Reiter,Ng_Viselter__amenability_LCQGs_coreps}]
\label{def:rep_amen}Let $\G$ be a LCQG. A representation $U$ of
$\G$ on a Hilbert space $\H$ is \emph{left amenable} if there is
a state $m$ of $B(\H)$ such that 
\begin{equation}
m\left[(\om\tensor\i)\left(U^{*}(\one\tensor x)U\right)\right]=\om(\one)m(x)\qquad(\forall x\in B(\H),\om\in\Lone{\G}),\label{eq:left_amen}
\end{equation}
in which case we say that $m$ is a left-invariant mean of $U$; equivalently,
the left action $\a_{U}:B(\H)\to\Linfty{\G}\tensorn B(\H)$ of $\G$
on $B(\H)$ given by $\a_{U}(x):=U^{*}(\one\tensor x)U$, $x\in B(\H)$,
has an invariant mean.\emph{ Right amenability} is defined similarly
by replacing $U$ by $U^{*}$.
\end{defn}

Observe that $U$ is left amenable if and only if $\overline{U}$ is right amenable.

\begin{rem}
  Although in the above definition $m$ is \emph{not} assumed to be normal, condition \eqref{eq:left_amen} can be abbreviated as
\[
(\i\tensor m)\left[U^{*}(\one\tensor x)U\right]=m(x)\one\qquad(\forall x\in B(\H)),
\]
where the slice map $\i\tensor m$ is as defined in \citep{Neufang__amp_cb_slice_maps}; see \citep[Lemma 2.2]{Ng_Viselter__amenability_LCQGs_coreps} for a succinct account.
\end{rem}
A possibly stronger notion of amenability involves almost-invariant vectors as part \ref{enu:amen_prod_alm_inv_vec__1} of the next result shows.

\begin{prop}[{\citep{Bedos_Tuset_2003}}]
  \label{prop:amen_prod_alm_inv_vec}Let $\G$ be a LCQG and $U,V$ be representations of $\G$. 
  \begin{enumerate}
  \item \label{enu:amen_prod_alm_inv_vec__1} If $V\tp U$ has almost-invariant vectors, then $U$ is left amenable and $V$ is right amenable. 
  \item \label{enu:amen_prod_alm_inv_vec__2} If $V\tp U$ has a non-zero invariant vector, then there exist a \emph{normal} left-invariant mean of $U$ and a \emph{normal} right-invariant mean of $V$.
  \end{enumerate}
  In both parts above, $V\tp U$ can equivalently be replaced by its contragradient $\overline{V\tp U} = \overline{V}\tpr \overline{U}$.
\end{prop}

\begin{proof}Assertion \ref{enu:amen_prod_alm_inv_vec__1} is precisely \citep[Proposition 5.2 (4)]{Bedos_Tuset_2003}. The same argument also gives \ref{enu:amen_prod_alm_inv_vec__2}. Indeed, if $U,V$ are representations of $\G$ on $\H,\KHilb$, respectively, and the unit vector $\Xi \in \KHilb \tensor \H$ is invariant under $V\tp U$, then $\omega_\Xi$ is clearly a (normal) left- and right-invariant mean of $V\tp U$. Thus, the normal state of $B(\H)$ given by $B(\H) \ni x \mapsto \omega_\Xi(\one \tensor x)$ is a left-invariant mean of $U$, and the normal state of $B(\KHilb)$ given by $B(\KHilb) \ni x \mapsto \omega_\Xi(x \tensor \one)$ is a right-invariant mean of $V$.
\end{proof}

We do not know in general whether the converse of \ref{enu:amen_prod_alm_inv_vec__1} is true; this would imply an affirmative answer to the famous amenability--co-amenability question. The answer is positive in the classical case by Bekka \citep{Bekka__amen_unit_rep}, and also in the discrete case:

\begin{thm}[{\citep[Theorem 9.5]{Bedos_Conti_Tuset__amen_co_amen_alg_QGs_coreps}}]
\label{thm:DQG_amen_prod_alm_inv_vec}If $U$ is a representation
of a discrete quantum group $\G$, then $\overline{U}\tp U$ has almost-invariant
vectors if (and only if) $U$ is left amenable.
\end{thm}
We are ready to present the main result of this subsection.
\begin{thm}\label{thm:amnbl-chars}
\label{thm:rep_amen_inv_vect}Let $U$ be a representation of a LCQG
$\G$ on a Hilbert space $\H$. 
\begin{enumerate}
\item Consider the following conditions:
\begin{enumerate}
\item \label{enu:rep_amen_inv_vect__1a}the representation $U$ is left amenable: there exists a left-invariant mean of $U$;
\item \label{enu:rep_amen_inv_vect__1b}the representation $\overline{U}\tp U$ has almost-invariant vectors.
\end{enumerate}
Then \prettyref{enu:rep_amen_inv_vect__1b}$\implies$\prettyref{enu:rep_amen_inv_vect__1a}, and the converse holds if $\G$ is discrete.
\item The following conditions are equivalent:
\begin{enumerate}
\item \label{enu:rep_amen_inv_vect__2a}there exists a \emph{normal} left-invariant mean of $U$;
\item \label{enu:rep_amen_inv_vect__2b}the representation $\overline{U}\tp U$ has a non-zero invariant vector.
\end{enumerate}
\end{enumerate}
\end{thm}
The implications \prettyref{enu:rep_amen_inv_vect__1b}$\implies$\prettyref{enu:rep_amen_inv_vect__1a} and \prettyref{enu:rep_amen_inv_vect__2b}$\implies$\prettyref{enu:rep_amen_inv_vect__2a} hold by \prettyref{prop:amen_prod_alm_inv_vec}, so we are interested only in the converse implications. The implication \prettyref{enu:rep_amen_inv_vect__1a}$\implies$\prettyref{enu:rep_amen_inv_vect__1b} for discrete quantum groups is precisely \prettyref{thm:DQG_amen_prod_alm_inv_vec}, but the implication \prettyref{enu:rep_amen_inv_vect__2a}$\implies$\prettyref{enu:rep_amen_inv_vect__2b} proved below is new. We will establish the last two implications in a unified way. Remark that our proof of \prettyref{enu:rep_amen_inv_vect__1a}$\implies$\prettyref{enu:rep_amen_inv_vect__1b} is much simpler than that of \prettyref{thm:DQG_amen_prod_alm_inv_vec} in \citep{Bedos_Conti_Tuset__amen_co_amen_alg_QGs_coreps}. We require the following lemma; in the discrete case it was given a different proof in \citep[Lemma 4.13]{Daws_Skalski_Viselter__prop_T}.
\begin{lem}[{\citep[Proposition 4.2]{Vaes__unit_impl_LCQG}, see \citep[Corollary 2.6.3]{Vaes__PhD}}]
  \label{lem:alpha_U_unit_impl}Let $U$ be a representation of a LCQG $\G$ on a Hilbert space $\H$. Consider the left action $\a_{U}$ of $\G$ on $B(\H)$ defined in \prettyref{def:rep_amen}. Then viewing $B(\H)$ as standardly represented on $\overline{\H}\tensor\H$, the canonical unitary implementation of $\a_{U}$ is $\overline{U}\tp U$.
\end{lem}
\begin{proof}[Proof of \prettyref{thm:rep_amen_inv_vect}]
  We prove the implications \prettyref{enu:rep_amen_inv_vect__1a}$\implies$\prettyref{enu:rep_amen_inv_vect__1b} (assuming that $\G$ is discrete) and \prettyref{enu:rep_amen_inv_vect__2a}$\implies$\prettyref{enu:rep_amen_inv_vect__2b}.  Condition \prettyref{enu:rep_amen_inv_vect__1a}, respectively \prettyref{enu:rep_amen_inv_vect__2a}, means that $\a_{U}$ has an invariant state, respectively a normal invariant state. Therefore, \citep[Proposition 4.11 (a)]{Daws_Skalski_Viselter__prop_T}, respectively \prettyref{thm:inv_norm_pos_func}, imply that the canonical unitary implementation of $\a_{U}$, which is $\overline{U}\tp U$ by \prettyref{lem:alpha_U_unit_impl}, has almost-invariant vectors, respectively a non-zero invariant vector.
\end{proof}

\section{Finite-covolume closed quantum subgroups and lattices}\label{sec:fin_covol_lat}

In this section we apply the preceding material and discussion on invariant weights to the study of closed quantum subgroups of finite covolume and lattices. 

\subsection{Finite covolume and unimodularity}

Classically, if a homogeneous space $G/H$ of a locally compact group admits a finite invariant measure then the unimodularity of $H$ is {\it equivalent} to that of $G$, i.e.~\prettyref{cor:g-so-h-unim} can be reversed when the invariant measure is finite. This follows for instance from the proof of \cite[Proposition B.2.2]{Bekka_de_la_Harpe_Valette__book}, which applies to finite-covolume $H\le G$ in general (rather than just {\it discrete} $H$, as the statement is phrased). In the present subsection we prove a quantum version of this remark.

\begin{thm}\label{thm:finite-covol-unim}
	Let $\bG$ be a LCQG and $\bH$ be a closed quantum subgroup in the sense of Vaes such that the left action of $\bG$ on $\Linfty{\G/\HH}$ admits an invariant normal state. Then, $\bH$ is unimodular if and only if $\bG$ is.
\end{thm}
\begin{proof}
	\prettyref{cor:g-so-h-unim} already deduces that $\bH$ is unimodular if $\bG$ is, so we are only concerned with the opposite implication. We thus assume that $\bH$ is unimodular and seek to show that $\bG$ is.
	
	The existence of a (necessarily faithful) invariant normal state $\theta$ on $\li{\bG/\bH}$ and the unimodularity of $\bH$ imply, via \prettyref{thm:quot_inv_nsf_weight}, that the action \prettyref{eq:ar-action} satisfies
	\begin{equation*}
	\a_r(\delta_{\G}^{it})=\delta_{\G}^{it}\tensor \one\qquad(\forall t\in\R),
	\end{equation*}
	that is, $\delta_{\bG}^{it}\in \li{\bG/\bH}$ for all $t\in \bR$. 
	
	By $\delta_{\G}$ being group-like and the invariance of $\theta$, we have
	\[
	\theta(\delta_{\bG}^{it}) \delta_{\bG}^{it} = (\i \tensor \theta)(\Delta_{\bG/\bH}(\delta_{\bG}^{it})) = \theta(\delta_{\bG}^{it}) \one \qquad (\forall t\in \bR).
	\]
	Furthermore, $\theta$ is normal, so that $\theta(\delta_{\bG}^{it}) \xrightarrow[t \to 0]{}  \theta(\one)=1$. Therefore, there is a neighborhood $I$ of $0$ in $\R$ such that for each $t \in I$ we have $\theta(\delta_{\bG}^{it}) \neq 0$, hence $\delta_{\bG}^{it} = \one$. This implies that $\delta_{\bG} = \one$, i.e., $\bG$ is unimodular.
\end{proof}

\begin{rem}Another way to complete the above proof after showing that $\delta_{\bG}^{it}\in \li{\bG/\bH}$ for all $t\in \bR$ is as follows.
	The von Neumann subalgebra $N$ of $\li{\bG}$ generated by the group-like unitaries $\delta_{\bG}^{it}$, $t\in \bR$ is a Baaj--Vaes subalgebra in the sense of, say, \cite[\S 2.1]{Kasprzak_Khosravi_Soltan__int_act_QGs} (terminology inspired by \cite{bv-cross}): a von Neumann subalgebra invariant under
	\begin{itemize}
		\item the co-multiplication (because each $\delta_{\bG}^{it}$ is group-like);
		\item the unitary antipode;
		\item the scaling group (these last two by, say, \cite[Proposition 7.12]{Kustermans_Vaes__LCQG_C_star}).
	\end{itemize}
	It follows from \cite[Proposition A.5]{bv-cross} that $(N,\Delta_\G|_N)$ is a LCQG $\bK$. The latter is classical and abelian because $N$ is abelian and $\bK$ is co-commutative, and in fact must be a subgroup of $\bR$ because $N$ is generated by a one-parameter group of group-like unitaries. Furthermore, $\bK$ must be compact: indeed, the embedding
	\begin{equation*}
	N\subseteq \li{\bG/\bH}
	\end{equation*}
	ensures that the $\bG$-invariant state on the latter restricts to an invariant state on $N$, which must thus be the left Haar state.
	
	It follows that $\bK$ is trivial, i.e.~$\delta_{\bG}^{it}=\one$ for all $t\in \bR$. In short, $\bG$ is unimodular.
\end{rem}

\begin{rem}
	Note that unimodularity does not necessarily lift from $\bH\le \bG$ to $\bG$ when the invariant measure on $\bG/\bH$ is infinite, even classically:
	
	If $\bG$ is the $ax+b$ group of, say, \cite[Example 15.17 (g)]{Hewitt_Ross__I} and $\bH<\bG$ is the (unimodular!)~subgroup of translations $x\mapsto x+b$ isomorphic to $(\bR,+)$, then the modular function of $\bG$ is given by $(\R \backslash \{0\}) \times \R \ni (a,b) \mapsto |a|^{-1}$, thus it restricts to that of $\bH$, and hence by \prettyref{thm:quot_inv_nsf_weight} there is a necessarily infinite $\bG$-invariant measure on $\bG/\bH$, even though $\bG$ is not unimodular.
\end{rem}

In particular, in the same spirit as \prettyref{cor:norm-unim}, we have:

\begin{cor}
	Let $\bH\trianglelefteq\bG$ be a normal closed quantum subgroup such that $\bG/\bH$ is compact. Then, $\bH$ is unimodular if and only if $\bG$ is.
\end{cor}
\begin{proof}
	An immediate consequence of \prettyref{thm:finite-covol-unim} and  \prettyref{cor:norm-com-inv}, the latter arguing that the Haar state $\varphi_{\bG/\bH}$ of $\bG/\bH$ is $\bG$-invariant.
\end{proof}

\subsection{Lattices}
We introduce lattices in LCQGs by direct analogy to the classical case discussed in \cite[Definition B.2.1]{Bekka_de_la_Harpe_Valette__book}.

\begin{defn}\label{def:lat}
  Let $\bG$ be a LCQG. A {\it lattice} in $\bG$ is a discrete closed quantum subgroup $\bH\le \bG$ in the sense of Woronowicz such that the left action of $\bG$ on $\Linfty{\G/\HH}$ has a (necessarily faithful) invariant normal state. 
\end{defn}

Purely quantum examples arise from the general theory of Drinfeld doubles, as introduced in \cite[\S 4]{pw} and studied amply afterwards, e.g.~in \cite{mnw,dfy,roy-XXr,roy-wyw,Arano__unit_sph_rep_Drinfeld_dbls,vm}. Here we will follow \cite{pw} (which uses the \emph{right} regular representation). The initial data to be fed into the general construction in \cite{pw} is a {\it compact} quantum group $\bG$. The underlying C$^*$-algebra $\Cz{\cD\bG}$ of the Drinfeld double of $\bG$ is a C$^*$-completion of the non-unital $*$-algebra
\begin{equation}\label{eq:alg-dual}
  \cO(\bG)\otimes \Cc{\widehat{\bG}},
\end{equation}
where $\cO(\bG)\subseteq C(\bG)$ is the unique dense Hopf $*$-subalgebra and $\Cc{\widehat{\bG}}$ (standing for functions with {\it compact} support) is the algebraic direct sum of the matrix algebras $M_{\alpha}$, each dual to the coefficient matrix coalgebra $C_{\alpha}\subseteq\cO(\bG)$ of an irreducible $\bG$-representation $\alpha$.

How the construction leads to a LCQG in the sense of \cite{Kustermans_Vaes__LCQG_C_star,Kustermans_Vaes__LCQG_von_Neumann,Kustermans__LCQG_universal} is explained briefly in \cite[\S 6]{dfy}: \prettyref{eq:alg-dual} is a $*$-algebraic quantum group \cite{vd-alg,vd-drab} and hence its reduced and universal analytic counterparts are constructible as in \cite{kust-alg-red,kust-alg-univ}. Alternatively, that is proved directly in \cite[\S 8]{mnw} (C$^*$-algebraic setting) and in \cite{bv-cross} (von Neumann algebraic setting). In fact, we have
\begin{equation*}
\Cz{\cD\bG}\cong \Cz{\bG}\tensormin\Cz{\widehat{\bG}} \text{ and }\li{\cD\bG}\cong \li{\bG}\tensorn\li{\widehat{\bG}},
\end{equation*}
and $\cD\bG$ is unimodular with $h_\G\otimes \psi_{\widehat{\bG}}$ being the bi-invariant Haar weight on $\cD\bG$, where $h_\G := \varphi_\G = \psi_\G$ (see \cite[Theorem 4.2]{pw}, \cite[\S 8]{mnw} or \cite[\S 5]{bv-cross}).

This general framework realizes both $\bG$ and $\widehat{\bG}$ as closed quantum subgroups of $\cD\bG$ in the sense of Woronowicz: this follows from \cite[Theorem 4.3 and preceding discussion]{pw} using the construction of the universal  face of $\cD\bG$ by applying \cite{kust-alg-univ} to \eqref{eq:alg-dual}. To elaborate, at the $*$-algebraic quantum group level of \eqref{eq:alg-dual}, and thus also at the universal level, the strong quantum homomorphisms realizing these two closed quantum subgroups are just the slice maps with respect to the co-units $\epsilon_{\widehat{\G}}, \epsilon_\G$ at the suitable tensor legs, respectively.
As a result, the description of the co-multiplication given in \cite[equation (4.16)]{pw} together with \cite[equations (4.10) and (4.11)]{pw} imply, using \eqref{eq:right_quant_hom__comult}, that the right action $\alpha_r$ corresponding to the inclusion $\widehat{\bG}\le \cD\bG$ is
\begin{equation*}
  \alpha_r=\mathrm{id}\otimes \Delta_{\widehat{\bG}}: \li{\bG}\tensorn\li{\widehat{\bG}}\to \li{\bG}\tensorn\li{\widehat{\bG}}\tensorn \li{\widehat{\bG}}
\end{equation*}
(as this holds on \eqref{eq:alg-dual}). From this and the ergodicity of the co-multiplication it follows that we have 
\begin{equation}\label{eq:dg-mod-hat}
  \li{\cD\bG/\widehat{\bG}} = \li{\bG}\tensor \C\one \cong \li{\bG}. 
\end{equation}
Furthermore, this description of $\a_r$ and the left invariance of $\varphi_{\widehat{\bG}}$ make it clear that $\a_r$ is integrable. Consequently, \citep[Corollary 5.6]{Kasprzak_Khosravi_Soltan__int_act_QGs} implies that $\widehat{\bG}$ is actually a closed quantum subgroup of $\cD\bG$ in the sense of Vaes. A similar reasoning works for $\G$.

The following result explains how Drinfeld doubles of compact quantum groups fit into the present context of studying (finite) invariant measures on homogeneous spaces.

\begin{thm}\label{thm:dbl-inv}
  Let $\bG$ be a compact quantum group and $\cD\bG$ its Drinfeld double, as above. The following conditions are equivalent.
  \begin{enumerate}
  \item\label{enu:dbl1} $\bG$ is of Kac type;
  \item\label{enu:dbl2} $\widehat{\bG}\le \cD\bG$ is a lattice in the sense of \prettyref{def:lat};
  \item\label{enu:dbl3} The left action of $\cD\bG$ on $\li{\cD\bG/\widehat{\bG}}$ has a completely invariant n.s.f.~weight.
  \end{enumerate}
\end{thm}
\begin{proof}
  Item \prettyref{enu:dbl2} is clearly formally stronger than \prettyref{enu:dbl3}, since the former asks that the left action of $\cD\bG$ on $\li{\cD\bG/\widehat{\bG}}$ admit a {\it finite} invariant n.s.f.~weight.

  To see that \prettyref{enu:dbl3}$\implies$\prettyref{enu:dbl1} recall that $\cD\bG$ is always unimodular. It then follows from \prettyref{cor:g-so-h-unim} and \prettyref{enu:dbl3} that $\widehat{\bG}$ too is unimodular, equivalently: of Kac type, being discrete. In turn, this implies that $\bG$ is of Kac type.

  It remains to argue that \prettyref{enu:dbl1}$\implies$\prettyref{enu:dbl2}, i.e.~that in the Kac case, the left action of $\cD\bG$ on $\li{\cD\bG/\widehat{\bG}}$ admits an invariant normal state. We will prove that the Haar state $h_\G$ of $\G$ satisfies this. Recalling the construction of \cite{pw}, we have 
 \begin{equation}
  \Delta_{\cD\bG} = (\i \tensor (\Ad{u}\circ \sigma ) \tensor \i) \circ (\Delta_\G \tensor \Delta_{\widehat{\bG}}), \label{eq:comult_Drinfeld_double}
 \end{equation} 
 where $u \in \M{\Cz{\widehat{\bG}} \tensormin \Cz{\G}}$ is the right regular representation of $\G$. Since $\G$ is of Kac type, we have \begin{equation}((\i \tensor h) \circ \Ad{u})(\one \tensor a) = h(a)\one \qquad (\forall a \in \Linfty{\G});\label{eq:Izumi}\end{equation} indeed, \cite[Corollary 3.9]{Izumi__non_comm_Poisson} says that \[((\i \tensor h) \circ \Ad{u^*})(\one \tensor a) \in \C \one \qquad (\forall a \in \Linfty{\G}),\] which, by using the boundedness of $R, \widehat{R}$ and the identities $h \circ R = h$ and $(\widehat{R} \tensor R)(u)=u$, is seen to be equivalent to \[((\i \tensor h) \circ \Ad{u})(\one \tensor a) \in \C \one \qquad (\forall a \in \Linfty{\G}),\] and applying (the extension to $\M{\Cz{\widehat{\bG}}} = \Linfty{\widehat{\bG}}$ of) $\widehat{\epsilon}$ now gives \eqref{eq:Izumi}. (Alternatively, \eqref{eq:Izumi} just follows from the computation in the proof of \cite[Corollary 3.9]{Izumi__non_comm_Poisson} by replacing the right regular representation with its adjoint.) Combining \eqref{eq:comult_Drinfeld_double}, \eqref{eq:Izumi} and the invariance of $h$ yields 
\[
\begin{split}
& ((\i \tensor \i \tensor h \tensor \i) \circ \Delta_{\cD\bG})(a \tensor \one) \\ & \qquad = 
((\i \tensor \i \tensor h \tensor \i) \circ (\i \tensor \Ad{u} \tensor \i)) (\Delta_\G(a)_{13}) 
 = h(a)\one 
\end{split} \qquad (\forall a \in \Linfty{\G}).
\]
Remembering that left action of $\cD\bG$ on $\li{\cD\bG/\widehat{\bG}}$ is just the restriction of $\Delta_{\cD\bG}$ and using \eqref{eq:dg-mod-hat}, we get the desired conclusion. 
\end{proof}


\subsection{Lattices and property (T)}

Classically, it is well known that property (T) transfers between locally compact groups and their finite-covolume closed subgroups, and in particular lattices: see e.g.~\cite[Theorem 1.7.1]{Bekka_de_la_Harpe_Valette__book}. In the quantum setup discussed here we first prove the following quantum version of the (ii)$\implies$(i) implication of that result, via what essentially amounts to a straightforward adaptation of the proof (modulo some paraphrasing).

\begin{thm}\label{thm:covol}
  Let $\bH\le \bG$ be a closed quantum subgroup of a LCQG in the sense of Woronowicz such that the left action of $\bG$ on $\Linfty{\G/\HH}$ has an invariant normal state. If $\bH$ has property (T), then so does $\bG$. 
\end{thm}
\begin{proof}
  Let $U\in \li{\bG} \tensorn B(\H)$ be a representation of $\bG$ with almost-invariant vectors witnessed by a net $(\z_i)_{i \in \mathcal{I}}$ of unit vectors. Denote by $P$ the projection of $\H$ onto the subspace $\Inv(U|_{\HH})$ of $\bH$-invariant vectors. 

  The net $(\z_i)_{i \in \mathcal{I}}$ is also almost-invariant for the restriction $U|_{\bH}$, and hence, since $\bH$ has property (T),
  \begin{equation}\label{eq:zetas}
    \|\zeta_i-P\zeta_i\|\to 0. 
  \end{equation}

  Let $\om$ be as in \prettyref{lem:hom_sp_inv_norm_state_inv_vects}. Then 
  \begin{equation}\label{eq:oPzeta}
  (\om\tensor\i)(U)P\zeta_i
  \end{equation}
  belongs to $\Inv(U)$ for all $i \in \mathcal{I}$. We will thus be done if we prove that it must be non-zero for sufficiently large $i$.

  To that end, note first that by \prettyref{eq:zetas} the vectors \prettyref{eq:oPzeta} are arbitrarily close in norm to
  \begin{equation}\label{eq:ozeta}
    (\om\tensor\i)(U)\zeta_i
  \end{equation}
  for large $i$. In turn, because $(\z_i)_{i \in \mathcal{I}}$ is almost invariant, the vectors \prettyref{eq:ozeta} get arbitrarily close in norm to the unit vectors $\zeta_i$.
\end{proof}

In particular, we have:

\begin{cor}\label{cor:lift-t}
  If a lattice in a LCQG $\bG$ has property (T), then so does $\bG$.
\end{cor}

\begin{cor}\label{cor:lift-t-dbl-CQG}
	If $\bG$ is a compact quantum group whose discrete dual has property (T), then so does its Drinfeld double $\cD\bG$.
\end{cor}

\begin{proof}
	Discrete quantum groups with property (T) are automatically of Kac type \cite[Proposition 3.2]{Fima__prop_T}, so we can combine \prettyref{thm:dbl-inv} and \prettyref{cor:lift-t} to get the result.
\end{proof}

Note that \prettyref{cor:lift-t-dbl-CQG} provides a new way to construct examples of non-classical, non-compact, non-discrete LCQGs with property (T), which does not rely on a deep representation-theoretical study such as in \cite{Arano__unit_sph_rep_Drinfeld_dbls}.
For instance, the Drinfeld doubles of the compact duals of the discrete quantum groups that were shown in \cite{Fima_Mukherjee_Patri_compact_bicross_prod,Vaes_Valvekens__prop_T_DQG_subfactor_tri_rep} to have property (T) also have property (T). 

The converse of \prettyref{thm:covol} (which holds classically) is the subject of upcoming work. 


\section*{Acknowledgements}
We are grateful to the referee for his/her comments, which significantly improved the paper's readability.
We also thank Matthew Daws and Adam Skalski for their interest and remarks.

A.C.~was partially supported by NSF grants DMS-1801011 and DMS-2001128. M.B.~was partially supported by NSF grant DMS-1700267.

This work is part of the project QUANTUM DYNAMICS supported by EU grant H2020-MSCA-RISE-2015-691246 and Polish Government grant 3542/H2020/2016/2.


\bibliographystyle{amsalpha}
\bibliography{bibliography}

\end{document}